 \documentclass[leqno,12pt]{article}

\usepackage[margin = 2cm]{geometry}
\usepackage{amsmath,amssymb,amsthm}
\usepackage{mathrsfs}
\usepackage{multirow}
\usepackage{diagbox}
\usepackage{graphicx}
\usepackage{color}
\usepackage{hyperref}
\usepackage{natbib}
\usepackage{caption}
\usepackage{subcaption}
\usepackage{float}
\usepackage[table]{xcolor}
\newtheorem{theorem}{Theorem}[section]

\newtheorem{assumption}[theorem]{Assumption}

\newtheorem{remark}[theorem]{Remark}
\newtheorem{proposition}[theorem]{Proposition}

\newtheoremstyle{definition}{}{}{}{}{\bfseries}{}{ }{}
\theoremstyle{definition}

\theoremstyle{remark}

\def\th{\hat{\tau}_{j, \lambda}}
\def\t{\tau_j}
\def\vjs{v_j(s)}
\def\vjt{v_j(t)}
\def\vhs{\hat{v}_{j, \lambda}(s)}
\def\vht{\hat{v}_{j, \lambda}(t)}
\def\cx{c(s,t)}
\def\cxh{\hat{c}(\lambda, s,t)}

\numberwithin{equation}{section}

\renewcommand{\theequation}{\arabic{section}.\arabic{equation}}

\allowdisplaybreaks

\begin{document}

\title{Detecting structural breaks in eigensystems of functional time series}

\author{
{\small Holger Dette, Tim Kutta} \\
{\small Ruhr-Universit\"at Bochum} \\
{\small Fakult\"at f\"ur Mathematik}\\
{\small Bochum, Germany} \\
}

  \maketitle

\begin{abstract}
Detecting structural changes in functional data  is a prominent topic in statistical literature. However not all trends in the data are important  in applications, but only those of large enough influence. In this paper we address the problem of identifying relevant changes in the eigenfunctions and eigenvalues of covariance kernels of  $L^2[0,1]$-valued   time series. By self-normalization techniques  we  derive pivotal, asymptotically consistent tests for relevant changes in these characteristics of the second order structure and investigate their finite sample properties in a simulation study. The applicability of our approach is demonstrated analyzing    German annual temperature data.
\end{abstract}

Keywords: functional time series,   relevant changes, eigenfunctions, eigenvalues, self-normalization

\maketitle

\bigskip

%%%%% Section1: Introduction

\section{Introduction}
\label{intro}

The analysis of functional data has gained increasing attention during
the past decades, due to recent advances in computing  and data collecting technologies. This surging interest is testified by a rapidly expanding scope of new statistical methods, as reviewed in the monographs of  \cite{bosq2000},
\cite{RS2005},  \cite{FerratyVieu2010}, \cite{HK2012}  and  \cite{hsingeubank2015}.

Applications of functional data analysis  include such diverse topics  as   imaging,  meteorology,
genomics, and economics. The analytical link between these fields lies in modelling observations as random functions, whether they are temperature curves or stock prices.  
%If the measurements of $X_n$ are sufficiently dense an interpolation yields an almost flawless reconstruction, such that it is reasonable to
% conceptualize them as observed functions (see, for example, \cite{JRice2004}, \cite{hm2006}). 
While this approach facilitates the data's interpretation for users, it in exchange poses theoretical
challenges, since each observation is now %-ex hypothesi-
 an element of a complex function space.  

%Accordingly many statistical methods have been extended  from  finite to  infinite dimensional data, ranging 	from mean and covariance estimation to more elaborate models. In particular
%prediction of functional times series  has gained considerable attention in the literature (see  %  \cite{bescarste2000}, \cite{bosq2000}, \cite{antpapsap2006},
%  \cite{kargonat2008},  \cite{aneivieu2008} , \cite{anecaofer2010},  \cite{didkokzha2012}  among many others). For this purpose several parametric and nonparametric techniques
%have been proposed.

Consequently dimension reducing procedures play a key role in this field, as they 
 make functional data amenable to the methods of finite dimensional statistics. 
 Among these,  \textit{functional principal component analysis} (fPCA) has taken the most prominent position. As principal component analysis (PCA) for finite dimensions,
fPCA is based on  projecting the data on  linear subspaces that explain  most of  its
variance. These spaces are spanned by the eigenfunctions  of the
estimated covariance operator.   An overview of the mathematical aspects of this
procedure can be found in the monographs of \cite{RS2005} and \cite{HK2012} and in the survey of \cite{Shang2014ASO}. 
Recently,  \cite{adh}   used functional principal components  for prediction in functional autoregressive models,
\cite{SHANG2017184} constructed forecasts  with dynamic updating based on  functional principal component regression
and \cite{GAO2019232} proposed {\it dynamic} fPCA for forecasting mortality rates.
Obviously dimension reducing procedures rest upon the assumption of "stable" eigensystems, i.e. that most of the variance of all data is confined to the same, low-dimensional subspace. 
This insight has furnished interest in methods to validate this assumption.

One option to investigate the stability of the eigensystem  is change point analysis, where one is  monitoring  a  functional time series
 for structural breaks  in the corresponding covariance  operators.
 In particular,  \cite{aston2012}  % \cite{JARUSKOVA20131500} 
  and  \cite{stoehr2019}   develop a powerful methodology to detect
  changes in the covariance operator. Similarly changes in the cross-covariance operator have been investigated by \cite{RS2019}. 
 % 
 % While many authors have addressed related  problems by change point analysis (see \cite{AUE2009}, \cite{zhangetal2011}, \cite{fremdtetal2014},
% \cite{tsudnish2014},  \cite{shatewwen2016} or    \cite{buccwend2017}  % \cite{grokokreim2017}
% among many others).
%The main body of these results refer to changes in the mean or a parameter of the model and can be modified to investigate changes in the second order structure.
%
    However, with the specific focus on fPCA
 it might be reasonable to conduct a refined analysis and  to   search   directly for changes in the eigenfunctions and eigenvalues of the covariance operator. Yet,   not much literature can be found in this direction.

The present paper contributes to this problem  in two respects.  On the one hand we develop a new  statistical methodology
for change point analysis of the eigenvalues and  eigenfunctions of a covariance operator of a functional times series. While    a test for a  change  in the spectrum of the covariance operator
has already been proposed by  \cite{ARS2018}, we are - to the  best of our  knowledge -
not aware of any procedure addressing the  problem of change point detection in the eigenfunctions corresponding to a sequence of functional data.
On the other hand  -  in contrast to the cited work, which has its  focus on the ``classical hypotheses'' of  strict equality-
we propose tests for  ``precise''
hypotheses as introduced in \cite{bergdela1987}. 
This means that  we are NOT interested in asserting arbitrarily small differences in the eigensystem before and after the change versus the hypothesis of exact equality. Rather, we try to detect or reject changes of relevant magnitude. 

For  example, if $\{X_{n} \}_{n=1, \ldots , N} $ is a functional time series  and  $\tau_{1,n}$ denotes the maximal eigenvalue  of the covariance operator  of  
$X_{n}$,  this means that  - in  contrast to  \cite{ARS2018}  - we  do not consider the null hypothesis $\tau_{1,1} = \ldots  = \tau_{1,N} $,
but  develop   a  statistical  methodology to test      the  hypothesis of no relevant deviation of the eigenvalue before and after the change point, that is
$$
H_0: |\tau^{(1)}_{1}-\tau^{(2)}_{1} |^2\le \Delta_\tau  \quad vs.  \quad H_1: | \tau^{(1)}_{1}- \tau^{(2)}_{1}
|^2> \Delta_\tau.
$$
Here   $\tau^{(1)}_{1}:= \tau_{1,1} = \ldots = \tau_{1,\lfloor N \theta_{0 } \rfloor } \not =  \tau _{1,\lfloor N \theta_{0 }  \rfloor +1 } = \ldots = \tau_{1,N} =:\tau^{(2)} _{1}$
for some $\theta_{0 }\in (0,1)$ and $\Delta_\tau$  is a given threshold defined by the concrete application (under the null hypothesis the  change is considered as not relevant).
The hypotheses regarding the other eigenvalues and  eigenfunctions are formulated similarly (see equations \eqref{H0H1eigenvalue} and \eqref{H0H1eigenvector} for more details).

The consideration of relevant hypotheses in the context of change point detection has been  introduced in   \cite{detteWied2015} and is
motivated by the observation that in many applications one is  not interested in small changes  of a parameter. 
For example,
in forecasting of functional times series, it  is  not  reasonable
to use only part of the data if a structural break in an eigenvalue (or  eigenfunction) is detected, but the  difference
before and after the change is rather small. 
In this case discarding the  data before the change could admittedly reduce the prediction bias, but come at the cost of a substantially increased variance due to a smaller sample size
used for prediction.

Relevant hypotheses have   been considered in statistics to different degrees since the mid 20th century (see
\cite{Hodges1954}), and have been investigated intensively in biostatistics, where tests for ``bioequivalence'' of certain drugs have nowadays become
standard (see for example  \cite{wellek2010}). In the context of change point analysis for functional data
relevant hypotheses have recently  been considered
by \cite{dettkokaue2019}  for Banach-space valued random variables and  by  \cite{DKV2018}
in Hilbert spaces. The first named paper concentrates on inference regarding the mean functions
while  \cite{DKV2018}  developed tests for a relevant structural break in  the  mean function and in the covariance operator.
The detection of   structural breaks in the eigenvalues  and  eigenfunctions is a  substantially more difficult  problem
due to their  implicit definition and  statistical tests have mainly been developed for the two sample case
(see  \cite{zhang2015}, who consider  classical hypotheses  and  \cite{auedettrice2019}, who discuss relevant hypotheses).

The aim of the present work is to develop  statistical methodology for detecting relevant changes in the eigensystem of a functional time series.
In Section \ref{sec2} we introduce the testing problems,  define corresponding test statistics  and give the main
theoretical results. Typically in change point  problems of this type estimation of the  long run covariance structures is required, which is
nearly intractable in the present context, because it involves all eigenvalues and eigenfunctions of the covariance operators before and after the change point
(see, for example, \cite{dauxois1982asymptotic} or \cite{Hall2006} for an explicit representation of the estimated eigenvalues  and  eigenfunctions  in terms of the empirical
covariance operator). We propose a self-normalization approach which avoids this problem.
In Section \ref{sec3} we illustrate our approach by virtue of a
small simulation study, as well as the investigation of the German weather data. Finally, in Appendices \ref{appA} and \ref{appB} we
provide the proofs of our findings and also give some auxiliary results.

%%%%% Section2: Testing for relevant changes (Main Results)

\section{Testing for relevant changes}
\label{sec2}
In this section we provide a precise outline of the testing problems considered in this paper and
subsequently present the main theoretical results.
Let $L^2[0,1]$ denote the Hilbert space of square integrable functions $f:[0,1] \to
\mathbb{R}$  equipped with the common inner product
$$\big< f, g\big>:=\int_0^1 f(t) g(t) dt \quad f, g \in L^2[0,1]. $$
The corresponding norm is denoted by $\|\cdot\|$. Notice
that according to the induced metric, functions that differ only on a set of
Lebesgue mass $0$ are identified.

Now suppose we observe a sample of $N$ random functions $X_1,...,X_N \in
L^2[0,1]$, where for any $n \in \{1,...,N\}$
\begin{equation}
X_n:= \begin{cases}
X_n=X_n^{(1)}, & n \le \theta_0 N\\
X_n=X_n^{(2)}, & n > \theta_0 N.
\end{cases} \label{assL2}
\end{equation}

Here $(X_n^{(1)})_{n \in \mathbb{Z}}$ and $(X_n^{(2)})_{n \in \mathbb{Z}}$ are
stationary sequences of centered, random functions in $L^2[0,1]$ and $\theta_0 \in (0,1)$
is a constant of proportionality. The assumption of vanishing expectations is made  for the sake of a simple notation and all results presented in this paper hold in the case $\mathbb{E}X_1^{(i)}=\mu^{(i)}$ for some $\mu^{(i)} \in L^2[0,1]$ $(i=1,2)$.  For a more detailed discussion of this case see Remark \ref{remark}.
 A general definition of expectations of random functions in
$L^2[0,1]$ can be found in \cite{BDH2018}. However in the subsequent discussion
we will always assume that
\begin{equation}\label{h1a}
   \mathbb{E}\|X_1^{(i)}\|^2<\infty, \quad \textnormal{for} \,\, i=1,2,
   \end{equation}
   which implies that expectations can be defined point-wise  (compare
\cite{HK2012}, Section 2.2).

Under assumption \eqref{h1a} the covariance kernel $c^{(i)}$ of $X_1^{(i)}$ ($i=1,2$) is almost everywhere defined and given by
$$c^{(i)}(s,t)=\mathbb{E}[X_1^{(i)}(s)X_1^{(i)}(t)]. $$
Regarded as a   function with two arguments it  is an element of $L^2([0,1]\times
[0,1])$, the space of square integrable functions on the unit square, which can be isomorphically identified with the tensor product Hilbert space $L^2[0,1]
\otimes L^2[0,1]$ (for details see \cite{JW1980}). We will also denote the induced norm of this space  by $\| \cdot\|$, since it will always be clear from the context, which space we refer to.
By Mercer's theorem (see \cite{Koenig1986} p.145) the kernels $c^{(1)}$ and $c^{(2)}$ permit the $L^2$-expansions
$$c^{(i)}(s,t)= \sum_{j \ge 1} \tau_j^{(i)} v_j^{(i)}(s)v_j^{(i)}(t), \quad
(i=1,2),$$
where  $v_1^{(i)}, v_2^{(i)},...\in L^2[0,1]$ are the eigenfunctions of  the
integral operator

\begin{equation}
f \mapsto \int_0^1 c^{(i)}(\cdot, t) f(t) dt, \quad f \in L^2[0,1]   \label{integraloperator}
\end{equation}

and $\tau_1^{(i)},\tau_2^{(i)},...$ are the corresponding eigenvalues.  For simplicity of
reference we assume for some fixed  $p \in \mathbb{N}$ that the first $p+1$
eigenvalues are the largest and that they are arranged in descending order, i.e. $
\tau_1^{(i)}\ge ...\ge \tau_{p+1}^{(i)}$. Furthermore the sets of eigenfunctions are
supposed to form orthonormal bases of the space $L^2[0,1]$, which can always be
enforced by adding further, orthogonal functions (with corresponding eigenvalues
$0$).

Based on the sample of observations $X_1,...,X_N$ we want to investigate
relevant changes in the eigensystems corresponding to $c^{(1)}$ and $c^{(2)}$. More
precisely, for some $j \in \{1,...,p\}$ we test whether the difference of the $j$-th
eigenvalues $\tau_j^{(1)}$
and $\tau_j^{(2)}$ or the $j$-th eigenfunctions  $v_j^{(1)}$ and $v_j^{(2)}$ exceeds a predetermined threshold. To be precise,  we consider for a fixed index $j \in \{1,...,p\}$ the hypotheses
\begin{equation}H_0: (\tau_j^{(1)}-\tau_j^{(2)})^2\le \Delta_\tau \quad vs.  \quad H_1:(
\tau_j^{(1)}-\tau_j^{(2)})^2> \Delta_\tau \label{H0H1eigenvalue}
\end{equation}
and
\begin{equation}
H_0: \|v_j^{(1)}-v_j^{(2)}\|^2\le \Delta_v  \quad vs.  \quad H_1:\|v_j^{(1)}-v_j^{(2)}
\|^2> \Delta_v.  \label{H0H1eigenvector}
\end{equation}
Here $\Delta_\tau$ and $\Delta_v$ are prespecified constants, denoting the maximal values for which the distances between the eigenvalues and eigenfunctions  are still considered scientifically irrelevant. The particular choice of $\Delta_\tau$ and $\Delta_v$ depends on the concrete application.
Note also that for $\Delta_\tau=0$ or $\Delta_v=0$ the hypotheses (\ref{H0H1eigenvalue}) and (\ref{H0H1eigenvector}) reduce to the classical change point detection problems for eigenvalues and eigenfunctions respectively.

In order to decide whether a relevant change either in the eigenvalues  or in the eigenfunctions
 has occurred we first have to identify the change point $\theta_0$.

\subsection{Change point estimation} The change point estimator is constructed by the CUSUM principle and defined by
\begin{equation}
\hat{\theta}:= \frac{1}{N} \textnormal{argmax}_{N\varepsilon \le k \le  N (1-
\varepsilon)} f(k), \label{changepointestimator}
\end{equation}
where the function $f$ is given by
\begin{align*}
f(k):= \frac{k(N-k)}
{N^2} \,\int_0^1 \int_0^1 \Big[ \frac{1}{k}\sum_{n=1}^k X_n(s)X_n(t)
-\frac{1}{N-k}& \sum_{n=k+1}^N X_n(s)X_n(t) \Big]^2 dsdt.
\end{align*}

Note that in definition \eqref{changepointestimator} we confine the maximization of $f$ to a subset of $\{1,...,N\}$ to obtain stable estimators. In practice this restriction is not an issue and even very small values for $\varepsilon$ can be used in \eqref{changepointestimator}.
We refer to   Section \ref{sec33}), where we demonstrate the stability of the estimator with respect to the choice of the threshold $\varepsilon$ by means of a simulation study. Before we  proceed we specify the basic
assumptions required for the theoretical statements presented in this paper.

\begin{assumption}\label{assumption1}
Let $i=1,2$.
\begin{itemize}
\item[1.] The sequence of random functions $(X_n^{(i)})_{n \in \mathbb{Z}}$
consists of Bernoulli shifts, i.e.
$$X_n^{(i)} =g^{(i)}(\epsilon_n, \epsilon_{n-1},...)  $$
for some measurable, non-random function $g^{(i)}:S^\infty \to L^2[0,1]$ and i.i.d.
innovations $\{ \epsilon_n \}_{n \in \mathbb{Z}}$ with values in a measurable space  $S$.
\item[2.]  All random variables $X_n^{(i)}(t, \omega)$ are jointly measurable in $(t,
\omega)$ for all $n \in \mathbb{Z}$.
\item[3.] $\mathbb{E}X^{(i)}_0=0$ and $\mathbb{E}\|X^{(i)}_0\|^{4+
\delta}<\infty$, for some $\delta \in (0,1)$.
\item[4.]
The random functions $X^{(i)}_n$ are  $m$-approximable.
In other words: for all  $m \in \mathbb{N}$ there exist i.i.d.  copies
$\{\epsilon_{n, m, k}^{*}\}_{k \in \mathbb{Z}}$  of $\epsilon_0$, independent of
$\{\epsilon_j\}_{j \in \mathbb{Z}}$, such that the   $m$-dependent random
functions $X_{n,m}^{(i)}$
\begin{equation}
X_{n,m}^{(i)} := g^{(i)}(\epsilon_n,...,\epsilon_{n-m+1},   \epsilon_{n, m, n-m}^{*},\epsilon_{n, m, n-
m-1}^{*},... ) 
\label{m-dependentRV}
\end{equation}
satisfy
\begin{equation}\label{h2}
 \sum_{m \ge 1} (\mathbb{E}\|X^{(i)}_n-X^{(i)}_{n,m}\|^{4+\delta})^{1/\kappa}<
\infty,
\end{equation}
for some $\delta \in (0,1)$ and $\kappa>4+\delta$.
\end{itemize}
\end{assumption}

Note that these assumptions match those in \cite{Berkes2013}, who derived weak invariance principles for $m$-approximable sequences. However, the stronger summability condition \eqref{h2}  is imposed here, since we are not estimating
mean functions, but   covariance kernels. We now state a first result concerning the
convergence rate of the change point estimator defined in \eqref{changepointestimator}.
The proof follows by similar arguments as given in the proof of  Proposition 3.1 in
\cite{DKV2018}, which are  omitted for the sake of brevity.

\begin{proposition} \label{Proposition2}
Suppose that Assumption 2.1 holds and that $c^{(1)} \neq c^{(2)}$. If $\varepsilon < \min \{\theta_0, 1-\theta_0 \}$, then
$$\hat{\theta}=\theta_0+o_P(1/\sqrt{N}). $$
\end{proposition}

In the next step we partition the data into two subgroups $X_1,...,X_{N   \hat{\theta}}	
$ and $X_{N \hat{\theta}+1},...,X_N,$ which we then use to estimate the covariance kernels $c^{(1)} $ and $c^{(2)}$ respectively. To be precise, we define for $\lambda \in [0,1]$ the estimators
\begin{align}
\hat{c}^{(1)}(  \lambda, s, t)=& \frac{1}{\lfloor N \hat \theta  \lambda \rfloor} \sum_{n=1}^{\lfloor  N \hat \theta \lambda \rfloor} X_n(s)X_n(t)  \label{estimatedkernels1}\\
\hat{c}^{(2)}(  \lambda, s, t)=& \frac{1}{\lfloor (N- N \hat \theta)  \lambda \rfloor} \sum_{n= N\hat \theta  +1}^{N\hat{\theta}+\lfloor (N-N \hat \theta ) \lambda \rfloor}X_n(s)X_n(t) \label{estimatedkernels2}
\end{align}
and obtain $ \hat c^{(1)} (\cdot, \cdot) := \hat{c}^{(1)}(1, \cdot, \cdot)$ and $\hat c^{(2)} (\cdot, \cdot) :=\hat{c}^{(2)}(1, \cdot, \cdot)$ as estimators for $c^{(1)}$ and $c^{(2)}$ respectively. In \eqref{estimatedkernels1} and \eqref{estimatedkernels2} the quantity $\lambda \in [0,1]$ denotes a parameter used for the subsequent self-normalization. If the kernels are degenerate they are interpreted as $0$-functions.

 By Proposition \ref{Proposition2} the amount of
misspecified data is  small and therefore we expect the estimated kernels $\hat{c}^{(1)}(\cdot, \cdot)$  and $\hat{c}^{(2)}(\cdot, \cdot)$, and hence their eigensystems to be close to those of $c^{(1)} $
and $c^{(2)}$.

\subsection{Relevant changes in the eigenvalues}

We now proceed to construct a test for the hypothesis \eqref{H0H1eigenvalue} of a  relevant change in the $j$-th eigenvalue. For
this purpose we define  for $i=1,2$ and $\lambda \in [0,1]$ the eigenfunctions
\begin{equation}
\hat{v}_{1,  \lambda}^{(i)}, \hat{v}_{2,  \lambda}^{(i)},... \label{eigenvectors}
\end{equation}
and the eigenvalues
\begin{equation}
\hat{\tau}^{(i)}_{1, \lambda} \ge  \hat{\tau}^{(i)}_{2, \lambda} \ge... \label{eigenvalues}
\end{equation}
 of the estimates $\hat{c}^{(i)} (\lambda, \cdot,\cdot)$ (defined in  \eqref{estimatedkernels1} and \eqref{estimatedkernels2}). Finally we denote by 
 \begin{equation}\label{h3}
   \hat \tau^{(i)}_j := \hat \tau^{(i)}_{j,1} \ ; \quad \hat v^{(i)}_j := \hat v^{(i)}_{j,1} \quad \quad j=1,2,\ldots
 \end{equation}
 the eigensystems of the estimated covariance operators of the full samples 
 $X_{1}, \ldots , X_{\lfloor N \hat \theta \rfloor }$ ($i=1$) and  $X_{\lfloor N \hat \theta \rfloor + 1}, \ldots , X_{N}$ ($i=2$).
 Note that the eigenfunctions are only determined up to a sign.
 Additionally,  we define the estimated squared difference of the $j$-th eigenvalues by 
$$\hat{E}_{j,N}( \lambda):=(\hat{\tau}_{j, \lambda}^{(1)}-\hat{\tau}_{j,
\lambda}^{(2)})^2, \quad \lambda \in [0,1]. $$
In view of the testing problem in  \eqref{H0H1eigenvalue} the natural entity of
interest is the statistic
$$
    \hat{E}_{j,N}:=\hat E_{j,N}(1)  = \big ( \hat \tau^{(1)}_{j}  - \hat \tau^{(2)}_{j} \big)^2,
$$
where $\hat \tau^{(i)}_{j}$  is defined in \eqref{h3}.
The null hypothesis \eqref{H0H1eigenvalue} of no relevant change in the $j$-th eigenvalue is now rejected for large values of $\hat E_{j,N}$. To find critical values for such a test we  determine  the asymptotic distribution of $\hat E_{j,N}$, which presupposes the following standard identifiability assumption (see e.g. \cite{HK2012}, \cite{Hall2006}).

\begin{assumption} \label{assumption2}
The first $p+1$ eigenvalues of the covariance kernel $c^{(i)}$ satisfy $\tau_1^{(i)}>...>
\tau_{p+1}^{(i)}>0$ for $i=1,2$.
\end{assumption}

It will be shown in the Appendix that under the  Assumptions \ref{assumption1} and \ref{assumption2}   the
statistic $\hat E_{j,N}$ is asymptotically normal in the sense that
\begin{equation}\label{h1}
 \sqrt{N}\big\{ \hat E_{j,N} -E_{j}\big\}  \to_D \mathcal{N}(0, \sigma_E^2),
\end{equation}
where the symbol $\to_D$ denotes weak convergence,  $E_j:=(\tau_j^{(1)}-\tau_j^{(2)})^2$ is the squared (unknow) difference between the $j$-th eigenvalues of the kernels $c^{(1)}$ and $c^{(2)}$, and $\mathcal{N}(\mu, \sigma^2)$ denotes a normal distribution with mean $\mu$ and variance $
\sigma^2$. In particular, if $\sigma^2 = 0$ this distribution  is defined as the point measure with probability mass $1$ at the point $\mu$.  The
variance of the normal distribution in \eqref{h1}   can be decomposed as
\begin{equation}
\sigma_E^2=4|\tau_j^{(1)}-\tau_j^{(2)}|^2 \big[ (\sigma_E^{(1)})^2 \theta_0+ (\sigma_E^{(2)})^2 (1-\theta_0)\big] ,
\end{equation}
where $\sigma^{(1)}_E$ and $\sigma^{(2)}_E$ are  non-negative parameters depending in a very complicated way on the long run variance of the time series $\{X_n^{(1)}\}_{n \in \mathbb{Z}}$ and $\{X_n^{(2)}\}_{n \in \mathbb{Z}}$ and the spectrum and eigenfunctions of the covariance operators $c^{(1)}$ and $c^{(2)}$. A precise definition of these quantities can be found in equations \eqref{longrunvariance1e} and \eqref{longrunvariance2e} in   Appendix \ref{appA}. From these representations it can be seen that $\sigma^{(1)}_E$ and $\sigma^{(2)}_E$ are notoriously difficult to estimate. We
circumvent this problem using   self-normalization techniques. This concept has been introduced for change point detection in a seminal paper by \cite{shao2010} and since then been used by many
authors. While most of this literature  concentrates on classical change point problems, \cite{DKV2018} introduced a novel type of self-normalization for relevant hypotheses and used it to define a self-normalized test for a relevant change in the mean of a time series. In the following we will further develop  this concept to detect relevant changes in the spectrum. For this purpose we define a normalizing factor 
\begin{equation}
\hat{V}_{j,N}:= \Big[ \int_0^1\lambda^4 \big\{\hat E_{j,N}(\lambda)-\hat E_{j,N}( 1) \big\}^2 d\nu(\lambda)\Big]^{1/2}, \label{VNE}
\end{equation}
where $\nu$ is a probability measure on the interval $(0,1)$.
Even though the specific choice of $\nu$ in \eqref{VNE} is generally not influential, it
is numerically convenient to use a discrete measure in applications rather than some
mathematically more natural choice like the Lebesgue measure.

The next Proposition is the central building block to prove the feasability of the
normalization approach.

\begin{proposition}  \label{proposition2a}
Suppose that Assumptions \ref{assumption1} and \ref{assumption2} hold, that $\varepsilon < \min \{\theta_0, 1-\theta_0 \}$, that $
\sigma^{(i)}_E>0$ for $i=1,2$ and that $\tau_j^{(1)}\neq \tau_j^{(2)}$ for some $j \in \{ 1,\ldots,p\}$. Then the following
weak convergence holds
\begin{equation}
\sqrt{N}(\hat E_{j,N} -E_{j}, \hat{V}_{j,N}) \to_D  \Big(\sigma_E \mathbb{B}(1), \sigma_E \Big[\int_0^1
\lambda^2 (\mathbb{B}(\lambda)-\lambda \mathbb{B}(1))^2 \nu(d \lambda)
\Big]^{1/2}\Big), \label{VNlimit}
\end{equation}
  where $\mathbb{B}$ is a standard Brownian motion.
\end{proposition}
\noindent Combining the weak convergence in \eqref{VNlimit} with the continuous mapping Theorem yields that,
\begin{equation}
\frac{\hat E_{j,N}-E_{j}}{\hat{V}_{j,N}} \to_D W \label{statisticconvergence1}
\end{equation}
  where the random variable
\begin{equation}
W:=\frac{\mathbb{B}(1)}{\big\{\int_0^1 \lambda^2[\mathbb{B}(\lambda)-\lambda
\mathbb{B}(1)]^2 \nu(d \lambda)\big\}^{1/2}} \label{W}
\end{equation}
 is a pivot. Some quantiles of the distribution of $W$ can be found in Table \ref{table:1} (where $\nu$ is a discrete uniform distribution).

\begin{table}[h!]
\begin{center}
\begin{tabular}{|c |c c c|}
\hline
& $\alpha=0.01$ & $\alpha=0.05$ & $\alpha=0.1$ \\ [0.5ex]
\hline\hline
$K=20$& 16.479 & 9.895 & 7.097\\
\hline
$K=30$ & 16.248 & 9.925 & 7.149 \\
\hline
\end{tabular}\\[2ex]
\caption{\textit{$(1-\alpha)$-quantiles of the distribution of the statistic $W$ in
\eqref{W}, where $\nu$ is the uniform distribution on the set $\{l/K: l=1,...,K-1\}$.}}
\label{table:1}
\end{center}
\end{table}
 We can now construct an asymptotic level-$\alpha$-test rejecting the
null hypothesis in \eqref{H0H1eigenvalue}, whenever
\begin{align}
\frac{\hat E_{j,N} -\Delta_\tau}{\hat{V}_{j,N}}>q_{1-\alpha}, \label{test}
\end{align}
where $q_{1-\alpha}$ is the asymptotic $(1-\alpha)$-quantile of the distribution of
$W$. These considerations are summarized in the following theorem.

\begin{theorem} \label{theorem1}
Suppose that Assumptions \ref{assumption1} and \ref{assumption2} hold, that $\varepsilon < \min \{\theta_0, 1-\theta_0 \}$,  and that
$\sigma^{(i)}_E>0$ for $i=1,2$. Then for any $j \in \{ 1,\ldots,p \}$ the rejection rule \eqref{test} defines  an asymptotic level-$\alpha$ and consistent
test for the hypothesis \eqref{H0H1eigenvalue}, that is
\begin{equation}
\mathbb{P}\bigg( \frac{\hat E_{j,N}  -\Delta_\tau}{\hat{V}_{j,N}}>q_{1-\alpha} \bigg) \to \begin{cases}
0, &   E_j < \Delta_\tau\\
\alpha, &  E_j = \Delta_\tau\\
1, & E_j > \Delta_\tau .
\end{cases} \label{levelalpha}
\end{equation}
\end{theorem}

%For a heuristic proof of Theorem \ref{theorem1} we note that Proposition
%\ref{proposition2a}  implies
% $$\mathbb{P}\Big( \frac{\hat E_{j,N}-E_{j}}{\hat{V}_{j,N}}>q_{1-
%\alpha} \Big)\to \alpha,$$
% if $\tau_j^{(1)}\neq \tau_j^{(2)}$. Consequently, we have in this case
%\begin{equation}
% \frac{\hat E_{j,N} -\Delta_\tau}{\hat{V}_{j,N}}= \frac{\hat E_{j,N}-E}{\hat{V}_{j,N}}+\frac{E-\Delta_\tau}{\hat{V}_{j,N}} \approx_D W
%+\frac{E-\Delta_\tau}{\hat{V}_{j,N}}. \label{approx}
%\end{equation}
%
%By Proposition \ref{proposition2a} it follows that $\hat{V}_{j,N}=o_P(1)$. The
%remaining case of $\tau_j^{(1)}= \tau_j^{(2)}$ is comparatively simple and will be
%discussed during the rigorous proof provided in Appendix \ref{appA}.\\

It should be noted that by the same arguments as above a test can be constructed
for the hypothesis of 
a relevant difference in  the eigenvalues before and after the change point, that is
\begin{equation}
H_0: E_{j}=(\tau_j^{(1)}-\tau_j^{(2)})^2>\Delta_\tau  \quad vs.  \quad H_1:(\tau_j^{(1)}-\tau_j^{(2)})^2 \le \Delta_\tau. \label{H0H1eigenvector2}
\end{equation}

The corresponding test rejects if
\begin{equation}
\frac{\hat E_{j,N} -\Delta_\tau}{\hat{V}_{j,N}}<q_{\alpha} \label{equiv}
\end{equation}
and the same arguments as in the proof of Theorem \ref{theorem1} show that this decision rule defines a consistent and asymptotic level-$\alpha$-test, that is

$$\mathbb{P}\bigg( \frac{\hat E_{j,N} -\Delta_\tau}{\hat{V}_{j,N}}<q_{\alpha} \bigg) \to \begin{cases}
0, &  E_j   > \Delta_\tau\\
\alpha, & E_j=  \Delta_\tau\\
1, & E_j <  \Delta_\tau.
\end{cases} $$
The formulation of the hypothesis in the form \eqref{equiv} is useful if one wants to establish the similarity between the eigenvalues at a controlled type-I-error. Hypotheses of the form \eqref{H0H1eigenvector2}  are frequently investigated in biostatistics, in particular in bioequivalence studies (see, for example,  {\cite{wellek2010}}).

\subsection{Relevant changes in the eigenfunctions}

Similar techniques  as in the preceeding section can be employed in the analysis of
the hypothesis \eqref{H0H1eigenvector} of no relevant change in the $j$-th eigenfunction. This task is slightly more intricate,
as we are now dealing with $L^2[0,1]$-functions instead of real numbers.

Recall the definition of the estimated eigenfunctions in \eqref{h3}. As we have already noticed such functions are only determined up to a sign. Thus, to make comparisons   meaningful,
we always assume that the inner products $\big< v_j^{(1)},  v_j^{(2)}\big> $, $
\big< v_j^{(1)}, \hat{v}^{(1)}_{j, \lambda}\big> $, $ \big< v_j^{(2)}, \hat{v}^{(2)}_{j,
\lambda}\big>$ and $ \big< \hat{v}^{(1)}_{j, \lambda}, \hat{v}^{(2)}_{j, \lambda}
\big> $ are non-negative for all $\lambda \in [0,1]$. This assumption is solely made for the sake of a clear
presentation. It can be dropped if in the testing problem \eqref{H0H1eigenvector}
and in the subsequently presented test statistic  all occurring vector
distances $\|v-v'\|$ are replaced by the $\min(\|v-v'\|,\|v+v'\|)$.
This is also how the statistic should be understood in applications.

We  estimate the
squared difference
$D_{j}:=\|v_j^{(1)}-v_j^{(2)}\|^2$ by $\hat D_{j,N}:=\hat{D}_{j,N}(1)$, where the statistic $\hat D_{j,N}(\lambda)$ is defined by
$$
\hat D_{j,N}(\lambda)=\|\hat{v}_{j,\lambda}^{(1)}-\hat{v}_{j,\lambda}^{(2)}\|^2.
$$
Recall that for $\lambda \in [0,1]$  $\hat{v}_{j,\lambda}^{(1)}$ and  $\hat{v}_{j,\lambda}^{(2)}$  are
 defined  as the
eigenfunctions of the estimated covariance operators from the samples
$$X_1,...,X_{\lfloor N \hat \theta \lambda \rfloor} \quad \textnormal{ and} \quad  X_{ N
\hat \theta +1},...,X_{N \hat \theta +\lfloor (N- N \hat \theta)\lambda \rfloor},$$
respectively (see \eqref{eigenvectors}). We
also introduce the the normalizing factor 
\begin{equation}
\hat U_{j,N} :=\left( \int_0^1 \lambda^4 \Big\{ \hat D_{j,N}(\lambda)  -  \hat D_{j,N}(1)
\Big\}^2 \nu(d \lambda) \right)^{1/2}. \label{VN}
\end{equation}
We propose to
reject the null hypothesis of no relevant change  in the $j$-th eigenfunction in
\eqref{H0H1eigenvector}, whenever
\begin{align}
\frac{\hat  D_{j,N}-\Delta_v}{\hat U_{j,N}}>q_{1-\alpha}, \label{test2}
\end{align}
where $q_{1-\alpha}$ is the upper $\alpha$ quantile of the distribution of
$W$ defined in \eqref{W}.
The following result shows that this test has asymptotic level-$\alpha$ and
 is consistent. The proof can be found in Section \ref{proof_thm_2} in the appendix.

\begin{theorem} \label{theorem2}
 Suppose that Assumptions \ref{assumption1} and \ref{assumption2} hold, that $\varepsilon < \min \{\theta_0, 1-\theta_0 \}$, that $j \in \{1,\ldots,p  \}$
   and that the quantities
  $\sigma_D^{(1)}, \sigma_D^{(2)}$  defined in \eqref{longrunvariance1}    are positive. Then the test defined  in
\eqref{test2} has asymptotic level $\alpha$ and is consistent for the hypothesis
\eqref{H0H1eigenvector}, that is   
\begin{equation}
\mathbb{P}\Big( \frac{\hat D_{j,N}-\Delta_v}{\hat U_{j,N}}>q_{1-\alpha} \Big) \to \begin{cases}
0, &  D_{j}  < \Delta_v\\
\alpha, & D_{j}  = \Delta_v \\
1, & D_{j}  > \Delta_v.
\end{cases} \label{levelalpha2}
\end{equation}
\end{theorem}

We conclude this section with a brief remark, that extends our results to non-centered data. This  is  of particular importance in
 applications such as presented in Section \ref{subsec33}.

\begin{remark} \label{remark}
{\rm A careful inspection of the proofs in Appendix \ref{appA} shows that all results in this section remain true  if the sequences of random variables $(X_n^{(i)})_{n \in \mathbb{Z}}$ have non-zero expectation  $\mu^{(i)} \in L^2[0,1]$ for $i=1,2$. In this case the estimators of the covariance kernels in \eqref{estimatedkernels1} and \eqref{estimatedkernels2}  have to be modified as follows
\begin{align*}
\hat{c}^{(1)}(  \lambda, s, t):= &\frac{1}{\lfloor N \hat \theta  \lambda \rfloor} \sum_{n=1}^{\lfloor  N \hat \theta \lambda \rfloor} \big[X_n(s)-\hat{\mu}^{(1)}( s) \big] \big[X_n(t)-\hat{\mu}^{(1)}( t) \big], \\
\hat{c}^{(2)}(  \lambda, s, t):=& \frac{1}{\lfloor (N- N \hat \theta)  \lambda \rfloor} \sum_{n= N\hat \theta  +1}^{N\hat{\theta}+\lfloor (N-N \hat \theta ) \lambda \rfloor} \big[X_n(s)-\hat{\mu}^{(2)}( s) \big] \big[X_n(t)-\hat{\mu}^{(2)}( t) \big],
\end{align*}
where $\hat \mu^{(1)}$ and $\hat \mu^{(2)}$ denote the empirical mean functions of the samples $X_1,...,X_{\lfloor N \hat \theta \rfloor}$ and $X_{\lfloor N \hat{\theta}\rfloor +1},...,X_N$ respectively.}
\end{remark}

%%%%% Section3: Finite sample properties

\section{Finite sample properties}
 \label{sec3}

In this section we investigate the performance of the new tests  by means of a
small simulation study and illustrate potential applications in a data example. All simulations are based on $4000$ simulation runs. We are interested in a test of the hypothesis of no
relevant differences in the   eigenvalues and eigenfunctions as defined in \eqref{H0H1eigenvalue} and \eqref{H0H1eigenvector}, respectively.
%\begin{equation}H_0: (\tau_j^{(1)}-\tau_j^{(2)})^2\le \Delta_\tau  \quad vs.  \quad H_1:(
%\tau_j^{(1)}-\tau_j^{(2)})^2> \Delta_\tau   \label{hypo3.2}
%\end{equation}
%\begin{equation}
%H_0: \|v_1^{(1)}-v_1^{(2)}\|^2\le \Delta_\tau  \quad vs.  \quad H_1:\|v_1^{(1)}-v_1^{(2)}
%\|^2> \Delta_\tau \label{hypo3.1}
%\end{equation}
In the subsequent results the measure $\nu$ in the statistics $\hat{V}_{j,N}$ and $\hat U_{j,N}$  is the uniform measure on the points $1/20,
2/20,...,19/20$ (see Table \ref{table:1}, $K=20$, for the critical values of the
distribution of $W$). Furthermore we assume that the change point is located at $\lfloor N/2 \rfloor$, that is $\theta_0=1/2$.

 \subsection{Relevant changes in the eigenvalues} \label{subsec32}

We investigate the rejection probabilities of the test \eqref{test2} for the hypothesis of no relevant change in the first and second eigenvalue.
To generate data we assume that the observed functions are smoothed over the
real Fourier basis of order $T$, which is defined for odd $T$ as
\begin{equation}\label{ort}
\{f_1,...,f_{T}\}=\Big\{1, \sqrt{2}\sin(2 \pi x),...,
\sqrt{2}\sin( \pi(T-1)x),\sqrt{2}\cos(2 \pi x),...,\sqrt{2}\cos( \pi(T-1)x)
\Big\}.
\end{equation}
Following \cite{AUE2009} we set  $T=21$, even though higher dimensions are
feasable.

 We define the covariance kernels in terms of the Fourier basis as follows
$$c^{(1)} (s,t):=\sum_{k=1}^T  \tau_k f_k(s)f_k(t) \quad \textnormal{and} \quad
c^{(2)}(s,t):=(1-\sqrt{E})\sum_{k=1}^4  \tau_k f_k(s)f_k(t)+\sum_{k=5}^T  \tau_k f_k(s)f_k(t), $$
where   $\tau_k:=1/k^2$ for $k=1,...,T$   and   the parameter   $E$  varies in the interval $[0,1]$.  Obviously the squared difference of the $j$-th eigenvalues of $c^{(1)} $ and $c^{(2)}$ is
$E_j=E/j^4$ for $j=1,..,4$ and subsequently $0$. 
%Note that, even though we only display tests for differences in the first two eigenvalues, we downscale the first four, to prevent potential crossovers. If we shrunk e.g. just the first one, it would at some point be exceeded by the formerly second largest and so forth. 

As we have seen in Section \ref{sec2}, the square  $L^2$-distance between the kernels $c^{(1)}$ and $c^{(2)}$ is of importance    for the performance of the change point
 estimator \eqref{changepointestimator}. In the present case it is given by
$$\int_0^1 \int_0^1 \big[ c^{(1)}(s,t)-c^{(2)}(s,t) \big]^2 ds dt = E \sum_{k=1}^4 \tau_k^2   = 1.07875 \cdot E .$$
% which is an isotonic function of $E$.

The simulated data is generated by randomly sampling sets of Fourier coefficients according to the above kernels.  First
we  generate $(N+1)$ i.i.d. random  vectors $\boldsymbol{\epsilon}_n:=(\epsilon_1,...,
\epsilon_T)\sim \mathcal{N}(0, \mbox{diag}(\tau_1,...,\tau_T))$, $n=0,...,N+1$. To
introduce potential dependence, we define a matrix $\Psi\in \mathbb{R}^{T \times
T}$ with i.i.d. normally distributed entries $\Psi_{l,k} \sim \mathcal{N}(0,\psi)$
and consider  the coefficients
\begin{align*}
\mathbf{a}_n= &\frac{\boldsymbol{\epsilon}_n+\Psi\boldsymbol{\epsilon}_{n-1}}
{\sqrt{1+\psi}}, \quad n=1,...,N.
\end{align*}
For $n=\lfloor N \theta_0 \rfloor +1,...,N$ we downscale the first four components of $\mathbf{a}_n:=(a_{n,1}, \ldots, a_{n,T})^T$ by a factor $\sqrt{1-\sqrt{E}}$. Finally the  process $\{ X_n\}_{n=1,\ldots,N}$ is defined by
\begin{equation} \label{h4}
X_n(s)=\sum_{k=1}^T a_{n,k}f_k(s).
\end{equation}
 An immediate
calculation reveals that  for $n=1,...,\lfloor N \theta_0 \rfloor$ the random variable $X_n$
has covariance kernel $c^{(1)} $ and for $n=\lfloor N \theta_0 \rfloor+1,...,N$
the covariance kernel of $X_n$ is given by $c^{(2)}$. The dependence of the data is determined by the choice
of $\psi$. For $\psi=0$ we generate i.i.d.  data and for $\psi>0$ an $fMA(1)$-
process. In the later case we choose $\psi$, such that  $\mathbb{E}\|\Psi\|_{L_1}
=1$.  In  each  simulation run we use  a new  realization of $\Psi$ to generate  the complete sample $X_{1}, \ldots , X_{N}$.

In Figure 1 and 2 we display the rejection probabilities of the test \eqref{test} for the hypothesis \eqref{H0H1eigenvalue} of no relevant change in the first and second eigenvalue, with level  $\alpha=5\%$. The threshold $\Delta_\tau$   is given by $0.1$ for $j=1$ and $0.005$ for $j=2$ and the tuning parameter in the estimator \eqref{changepointestimator} is chosen as  $\varepsilon = 0.05$.
\\[-5ex]

\begin{figure}[H]
   \centering
  \begin{subfigure}{0.5
  \textwidth}
\centering
 \includegraphics[width=1\linewidth]{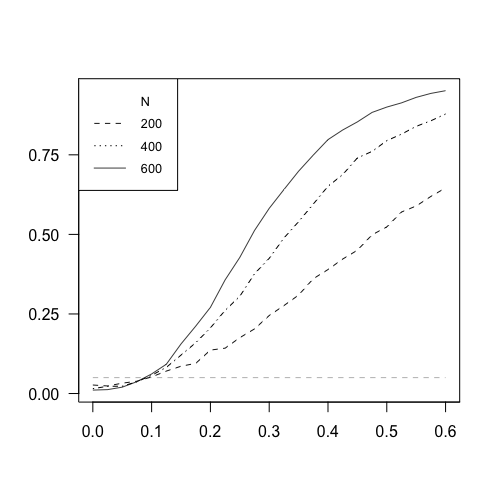}
 \label{fig:sub1}\\[-20ex]
 \end{subfigure}%
 \begin{subfigure}{0.5\textwidth}
 \centering
 \includegraphics[width=1\linewidth]{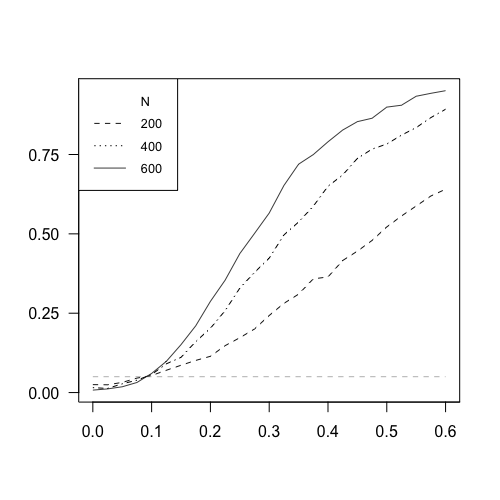}
\label{fig:sub2}
 \end{subfigure}\\[-8ex]
\caption{\textit{Rejection probabilities of the test \eqref{test} for the hypothesis \eqref{H0H1eigenvalue} of no relevant change in   the first eigenvalue, where $\Delta_\tau = 0.1$. Data comes from an   i.i.d. sequence (left) and an $fMA(1)$-process (right). The vertical dashed line visualizes the $5 \% $ level.  } \\[-8.5ex]}
 \label{fig4}
 \end{figure}
\begin{figure}[H]
 \begin{subfigure}{0.5\textwidth}
 \centering
\includegraphics[width=1\linewidth]{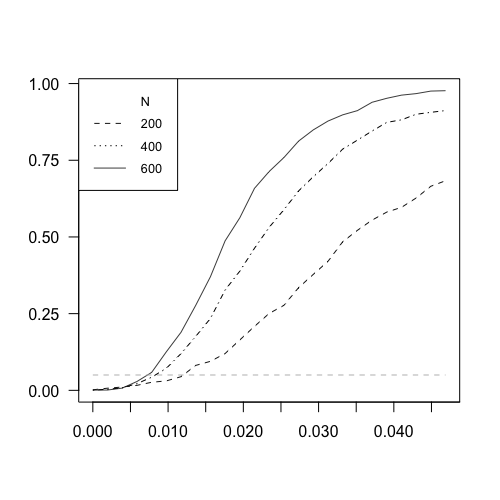}
   \label{fig:sub1}
 \end{subfigure}%
 \begin{subfigure}{0.5\textwidth}
\centering
\includegraphics[width=1\linewidth]{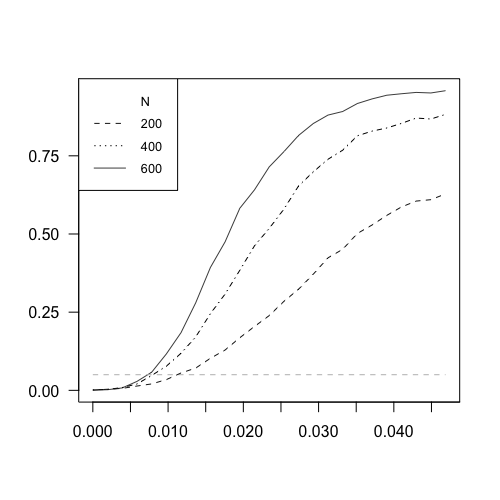}
\label{fig:sub2}
\end{subfigure}\\[-8ex]
\caption{\textit{Rejection probabilities of the test \eqref{test} for the hypothesis \eqref{H0H1eigenvalue} of no relevant change in   the second eigenvalue, where $\Delta_\tau = 0.005$. Data comes from an   i.i.d. sequence (left) and an $fMA(1)$-process (right). The vertical dashed line visualizes the $5 \% $ level.   } \\[-4ex]}
\label{fig5}
\end{figure}

 According to the theoretical discussion in Theorem \ref{theorem1} the  test should have rejection probabilities smaller, close and larger to $\alpha$
if  $ E_j <  \Delta_\tau $ (interior of the null hypothesis), $ E_{j}=\Delta_\tau$   (boundary of the null hypothesis) and $ E_j  > \Delta_\tau$ (alternative), respectively. 
For the first eigenvalue (Figure \ref{fig4}) we observe a good approximation of the nominal level, at 
the boundary of the null hypothesis
 even if the sample size is $N=200$ and a reasonable power. For the second eigenvalue (Figure \ref{fig5}) the test 
slightly conservative for $N=200$ at the boundary of the null hypothesis, but  the level is close to $\alpha $ for $N=400$ and $N=600$. A comparison   of the left and right part in
Figures \ref{fig4} and \ref{fig5}  shows  that   dependence in the data has only a small
impact on both type I and type II error, even though a subtle increase is visible.
Further simulations with different distributions of the Fourier coefficients
show that the results are stable in this respect, although  heavier tails
lead to a  loss of power. These results are not reported for the sake of brevity.

\subsection{Relevant changes in the eigenfunctions} \label{subsec31}
 To investigate the finite sample properties of the test \eqref{test2} for the hypothesis of no relevant change in the $j$-th eigenfunction in \eqref{ort} we define the covariance kernels
\begin{equation}
c^{(1)} (s,t):=\sum_{k=1}^T  \tau_k v_k^{(1)}(s)v_k^{(1)}(t) \quad \textnormal{and} \quad
c^{(2)}(s,t):=\sum_{k=1}^T  \tau_k v_k^{(2)}(s)v_k^{(2)}(t). \label{simkernels}
\end{equation}
Here the eigenvalues of $c^{(1)}$ and $c^{(2)}$ are the same, that is
$$\tau_k=1/k^2 \quad \textnormal{for}
\,\,k=1,...,T,$$
and the respective eigenfunctions $v_k^{(i)}$ are  defined by
$$v_k^{(1)}(s)=f_k(s) \quad (k=1,...,T), \quad v_k^{(2)}(s)=f_k(s) \quad (
k=3,...,T),$$ and
$$v_1^{(2)}(s)=\cos(\varphi)f_1(s)+\sin(\varphi)f_2(s), \quad v_2^{(2)}(s)=
\cos(\varphi)f_2(s)-\sin(\varphi)f_1(s),$$
where the Fourier basis $\{ f_1,\ldots, f_T \}$ is defined in \eqref{ort}.
 A simple calculation yields that the $L^2$-distance between the kernels is given by
$$\int_0^1 \int_0^1 \big[ c^{(1)}(s,t)-c^{(2)}(s,t) \big]^2 ds dt = \frac{5[1-\cos(\varphi)]}{2} $$
 and that the distance between the first and second eigenfunctions   is
 \begin{eqnarray*}
 \|v_1^{(1)}-v_1^{(2)}\| &=& \|v_2^{(1)}-v_2^{(2)}\| = \sqrt{2-2 \cos(\varphi)}.
 \end{eqnarray*}
By construction the two kernels $c^{(1)} $ and $c^{(2)}$ in this example are very similar,
and therefore the estimation of the change point  is a  challenging task. Any further difference in the eigensystems would increase the
$L^2$-distance of the kernels and thus facilitate the problem of change point detection.

The data  is generated  as described in Section \ref{subsec32}, where the coefficients $a_{n,k}$  in the representation
 \eqref{h4}  are given by  the vectors
\begin{align*}
\mathbf{a}_n= & (a_{n,1}, \ldots , a_{n,T})  = \frac{\boldsymbol{\epsilon}_n+\Psi\boldsymbol{\epsilon}_{n-1}}
{\sqrt{1+\psi}}, \quad n=1,...,\lfloor N \theta_0 \rfloor, \\
\mathbf{a}_n= & (a_{n,1}, \ldots , a_{n,T})  = R_{1,2}(\varphi) \frac{\boldsymbol{\epsilon}_n+\Psi
\boldsymbol{\epsilon}_{n-1}}{{\sqrt{1+\psi}}}, \quad n=\lfloor N \theta_0 \rfloor
+1,...,N.
\end{align*}
Here  $R_{1,2}(\varphi)$ is the rotation matrix of the first two components by an angle
$\varphi$. Note that for $n=1,...,\lfloor N \theta_0 \rfloor$ the random variable $X_n$ in \eqref{h4}
has covariance kernel $c^{(1)} $ and for $n=\lfloor N \theta_0 \rfloor+1,...,N$
the covariance kernel of $X_n$ is given by $c^{(2)}$.

In Figure 3 and 4 we display the rejection probabilities of the test \eqref{test2} for the hypothesis \eqref{H0H1eigenvector} of no relevant change in the first and second eigenfunction (threshold $\Delta_v= 0.1$), with level  $\alpha=5\%$. The tuning parameter in the change point estimator \eqref{changepointestimator}   is given by $\varepsilon = 0.05$.     \\[-6ex]

%%%%
\begin{figure}[H]
 \centering
\begin{subfigure}{0.5\textwidth}
 \centering
\includegraphics[width=1\linewidth]{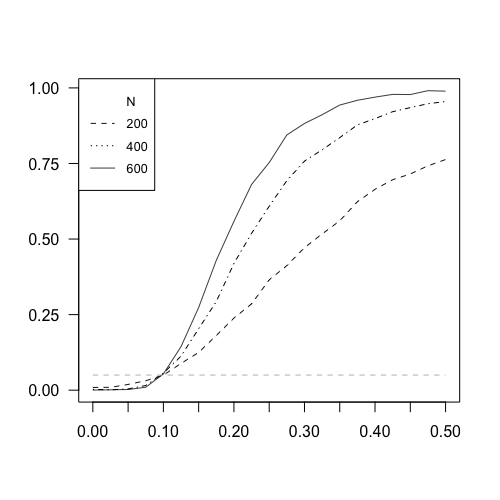}\\[-5ex]
\label{fig:sub11}
\end{subfigure}%
\begin{subfigure}{0.5\textwidth}
\centering
\includegraphics[width=1\linewidth]{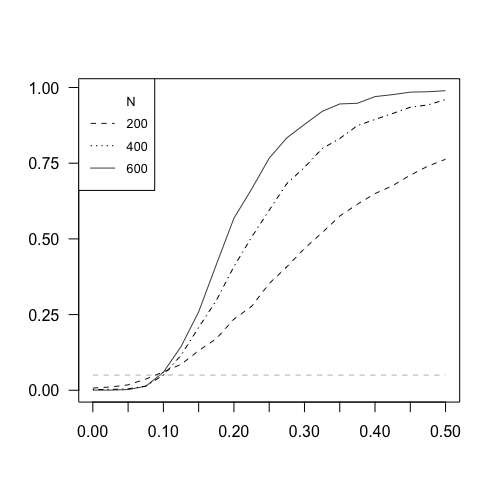}\\[-5ex]
 \label{fig:sub12}
\end{subfigure}
 \label{fig1}
\caption{\textit{Rejection probabilities of the test \eqref{test2} for the hypothesis \eqref{H0H1eigenvector} of no relevant change in the first eigenfunction. Data comes from an   i.i.d. sequence (left) and an $fMA(1)$-process (right).  The vertical dashed line visualizes the $5 \%$ level. }\\[-9ex]}
\end{figure}
\begin{figure}[H]
\begin{subfigure}{0.5\textwidth}
\centering
\includegraphics[width=1\linewidth]{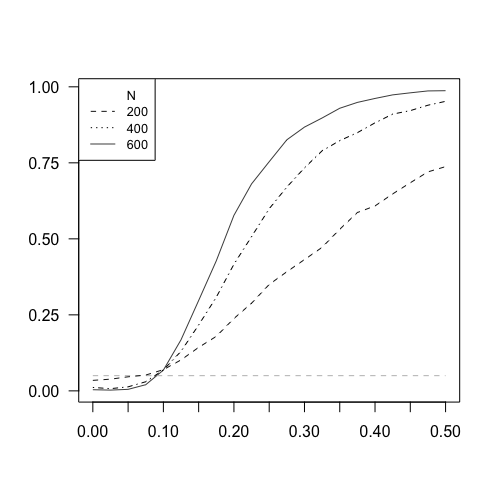}\\[-3ex]
\label{fig:sub21}
\end{subfigure}%
\begin{subfigure}{0.5\textwidth}
\centering
 \includegraphics[width=1\linewidth]{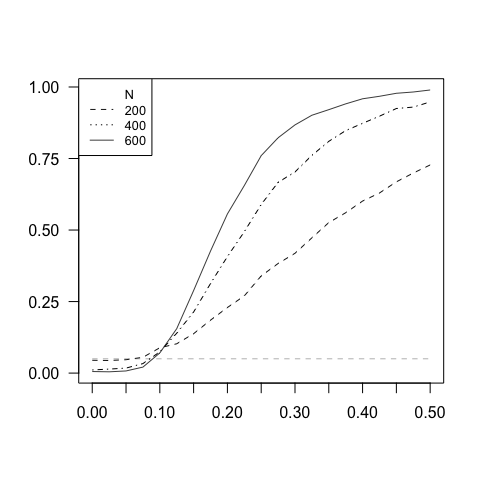}\\[-5ex]
\label{fig:sub22}
\end{subfigure}
\label{fig2}
\caption{\textit{Rejection probabilities of the test \eqref{test2} for the hypothesis \eqref{H0H1eigenvector} of no relevant change in  the second eigenfunction. Data comes from an   i.i.d. sequence (left) and an $fMA(1)$-process (right). The vertical dashed line visualizes the $5 \%$ level.  }\\[-1ex]}
\end{figure}

 We observe a good approximation of the nominal level  at boundary of the null hypothesis 
 ($D_{j} = \Delta_v$) and the test detects alternatives with decent power. Similar to the investigation of the eigenvalues, additional dependence has little impact on the results. An interesting difference occurs in the consideration of the second eigenfunction. Whereas for the second eigenvalue the self-normalized test \eqref{test} is slightly more conservative for small sample sizes, we observe that for the second eigenfunction the test \eqref{test2} becomes slightly more liberal in this case.

\subsection{Choice of the tuning parameter in  (\ref{changepointestimator})} \label{sec33}

It is of interest to investigate the impact of the choice of the boundary parameter $\varepsilon$ on the  performance of the tests. In practice the choice $\varepsilon=0.05$
indicates only  moderate knowledge, but one might want to use even smaller values, or
more rigorously put $\varepsilon=0$. We therefore consider the model of the preceding section, for the first eigenfunction with an i.i.d.  sample, of size 
$N=400$ and investigate the impact of   different   choices $\varepsilon=0, 0.005, 0.01, 0.05$ on the performance of the test \eqref{test2} for the hypothesis \eqref{H0H1eigenvector} of no relevant difference in the first eigenfunction. The results are displayed in Figure \ref{fig:sub21a} and indicate a high stability with respect to the choice of $\varepsilon$. Whereas the power of the test is hardly influenced by 
the choice of  $\varepsilon$, we observe that  the choice  $\varepsilon=0$ produces comparatively high rejection probabilities under the null hypothesis, particularly for small values of $D_{j}$. This effect can be explained as follows. Under the alternative the two samples $X_1,...,X_{\lfloor \theta_0 N\rfloor}$ and  $X_{\lfloor \theta_0 N\rfloor+1},...,X_N$  have  quite  distinct covariance structures and so the change point estimator will perform well, regardless of the choice of $\varepsilon$. 
However, if   $D_{j}$ is close to $0$, the samples will be be nearly indistinguishable, such that  the change point 
$\theta_0$ is estimated less  accurately. 
%If the very first or very last measured functions are outliers, they easily attract the change point estimator $\hat \theta$, leading to an extremely uneven segmentation of the data and hence to instability in the estimated eigenvectors. 
In such cases the test  has  larger  rejection probabilities. 

 This interpretation is visualized by  Figure \ref{fig:sub22b} where we display a histogram of the estimated change points using the estimator \eqref{changepointestimator} with $\varepsilon =0$    and true distance of the eigenfunctions $D_j=0$. We observe that in this case the estimator frequently localizes the change point in the first or last bin.
 For the  problem of testing for relevant differences in the eigenvalues as considered in Section \ref{subsec32},
similar effects can be observed, which are not reported to avoid redundancy.
  \\[-5ex]

 \begin{figure}[H]
\begin{subfigure}{0.5\textwidth}
\centering
\includegraphics[width=1\linewidth]{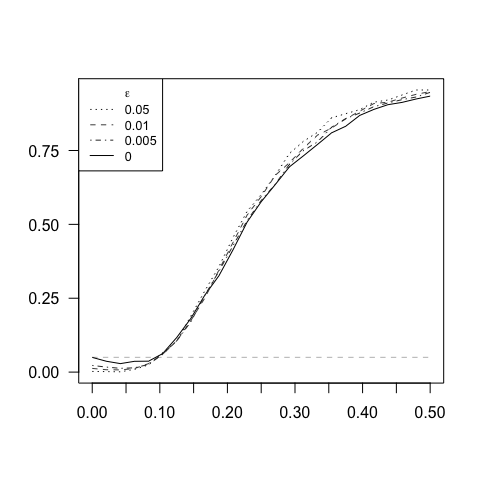}\\[-2ex]
\caption{ \textit{Simulated rejection probabilities of the \\ test \eqref{test2} for the hypothesis \eqref{H0H1eigenvector} of no \\ relevant change in the first eigenfunction  \\ ($\Delta_v = 0.1$). The  sample size is
$N=400$ and  \\  different choices of the tuning
parameter $\varepsilon$ \\ in the estimator \eqref{changepointestimator} are considered.}}
 \label{fig:sub21a}
 \end{subfigure}%
 \begin{subfigure}{0.5\textwidth}
 \centering
  \includegraphics[width=1\linewidth]{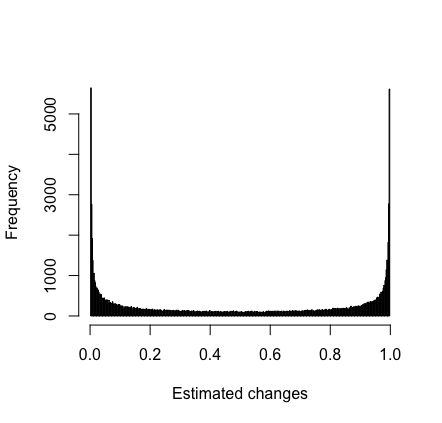}\\[-2ex]
\caption{\textit{Histogram of $100.000$ realizations of the change point estimator \eqref{changepointestimator} with tuning parameter $\varepsilon=0$. The sample size is $N=400$. \\[5ex]}}
 \label{fig:sub22b}
 \end{subfigure}
 \caption{$ $\\[-3ex] \label{fig3}}
 \end{figure}

 \subsection{A data example} \label{subsec33}

In this subsection we apply the  new methodology to identify relevant changes in   temperature measurements in northern Germany. The data consists of daily temperature averages, published by the national meteorological agency ``Der  Deutsche Wetterdienst'' ( \url{ https://www.dwd.de/DE/Home/home_node.html}) in the state of Bremen (Lat=$53.1^{\circ}$, Long$=008.7 ^{\circ}$) over the years $1890$ to $2018$. Due to incomplete data in the first years, as well as immediately after WWII, the years $1890-1893$ as well as $1945-1946$ have been removed from the data. This leaves us with $123$ years of daily measurements.

% Following \citep{Aue2015}
We smooth this data over the  Fourier basis defined in \eqref{ort}, for $T=41$, a choice between \cite{fremdtetal2014} (T=49) and \cite{AUE2009} (T=21), which allows a reasonable approximation of the daily temperatures and still reflects general trends. However it is worth noticing that our subsequent findings are quite robust with respect to different choices of $T$.

For the tuning in the change point estimator \eqref{changepointestimator} we choose  $\varepsilon=0.01$, which yields $N \hat \theta =92$ as an estimated change point. This corresponds to the year $1988$ and we can calculate the respective eigenfunctions and eigenvalues, before and after the change. Exemplarily we show in Figure \ref{fig6} the first three eigenfunctions before and after the estimated change point. We observe that the first eigenfunctions are quite similar, but larger differences can be found between the second and third eigenfunctions (see lower panel in Figure \ref{fig6}). \\[-5ex]

\begin{figure}[H]
\centering
\begin{subfigure}[t]{0.8\textwidth}
\centering
\includegraphics[width=\textwidth]{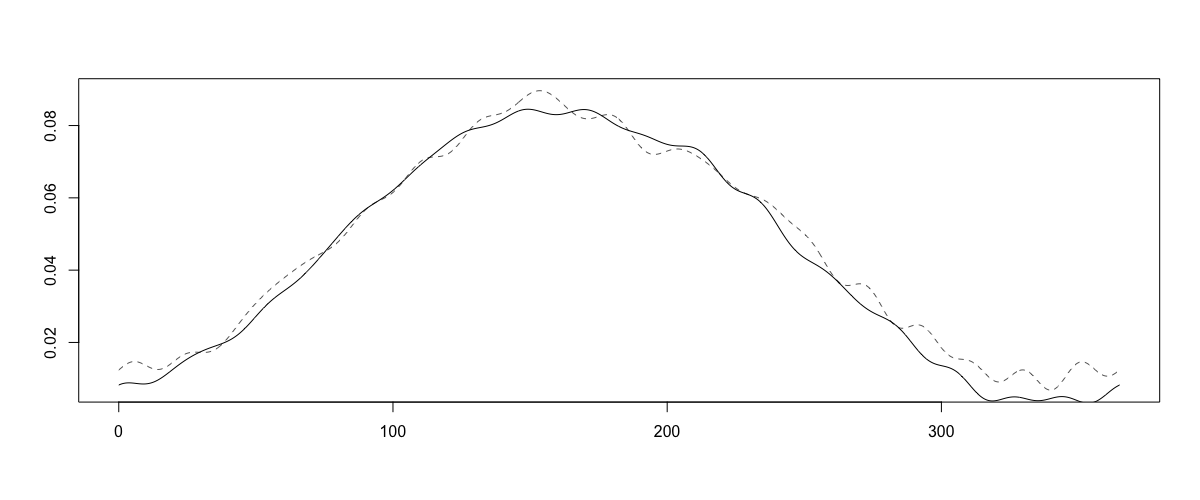}
\end{subfigure}

\begin{subfigure}[t]{0.8\textwidth}
\centering
\includegraphics[width=\textwidth]{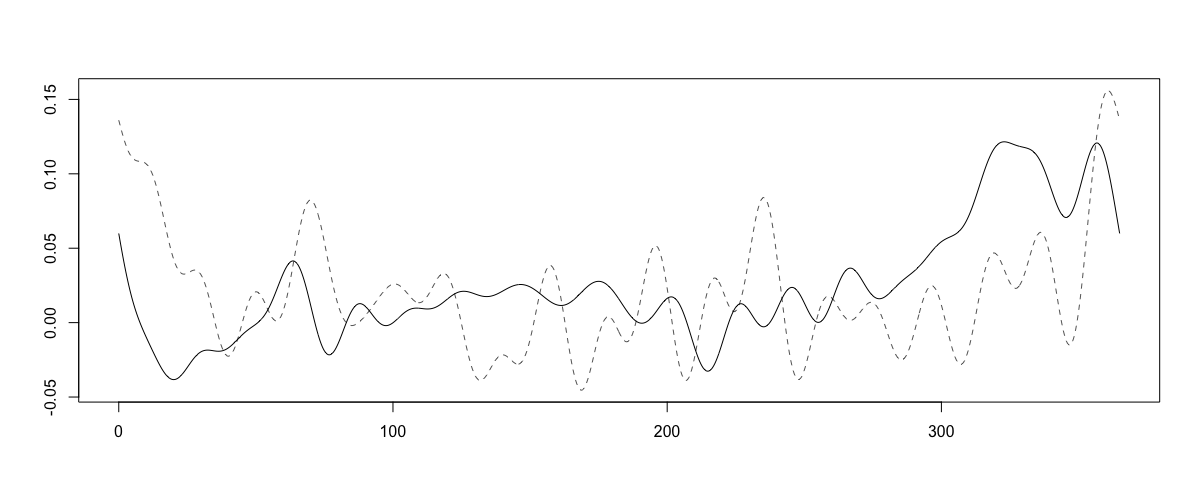}
\end{subfigure}

\begin{subfigure}[t]{0.8\textwidth}
\centering
\includegraphics[width=\textwidth]{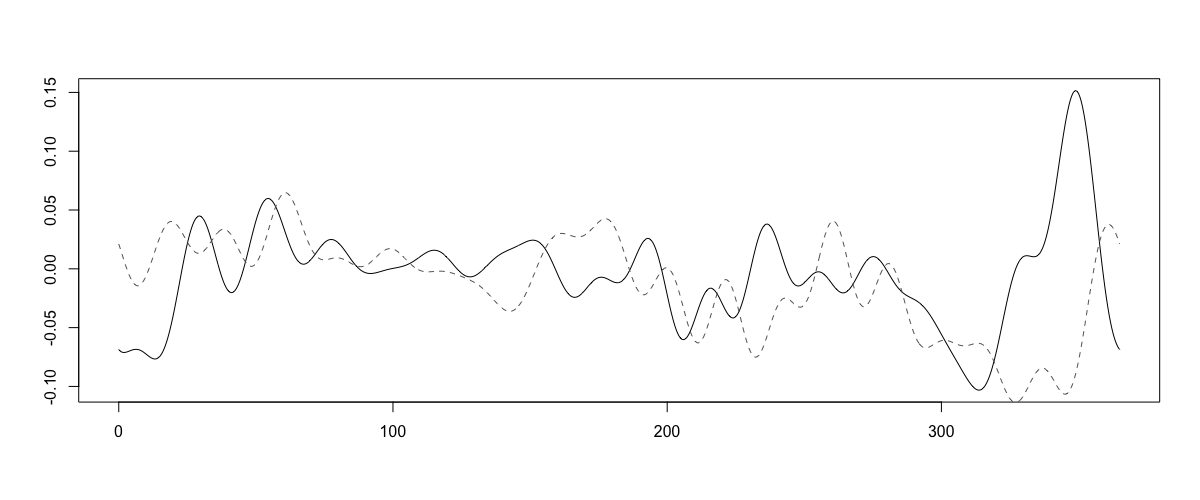}
 \end{subfigure}

 \caption{\textit{Eigenfunctions of the estimated covariance operators from 1894-1988 (solid) and 1989-2018 (dotted). Upper panel: First eigenfunctions $\hat v_{1}^{(1)}$ and $\hat v_{1}^{(2)}$. The curves differ only slightly. Is this relevant? Middle panel: Second eigenfunctions $\hat v_{2}^{(1)}$ and $\hat v_{2}^{(2)}$.  The functions are more dissimilar, reveiling very limited common trends. Lower panel: Third eigenfunctions $\hat v_{3}^{(1)}$ and $\hat v_{3}^{(2)}$. These functions are rather different, close to orthogonality.}\label{fig6}}

\end{figure}

For the first five eigenfunctions we have applied the test \eqref{test2} for the hypothesis \eqref{H0H1eigenvector} of no  relevant change to see  whether we can find relevant differences for varying sizes of $\Delta_v$. To make our results more interpretable we translate the measure of similarity $\Delta_v$ into an angle $\varphi$, i.e., if the squared distance of two eigenfunctions differs by at least $\Delta_v \approx 0.58$ the ``geometric angle'' between them is at least $\varphi =\pi/4$. This is again due to the fact that the angle $\varphi$ between two eigenfunctions $v, w$ determines their distance
$$\|v-w\|=\sqrt{2-2 \cos(\varphi)}.$$

In Table \ref{table:2} we  display the decisions of the test \eqref{test2} for the hypothesis \eqref{H0H1eigenvector} of no relevant changes in the eigenfunctions where different thresholds are considered.
We observe that the test does not detect relevant changes between  the first and the second eigenfunctions. In contrast, the eigenfunctions of larger order display significant differences, which
confirms the visualization in  Figure \ref{fig6}.
The test detects relevant changes in the third, fourth and fifth eigenfunctions for nearly all thresholds  (and the same holds true for eigenfunctions of larger order).

 \begin{table}[h!]
\begin{center}
\begin{tabular}{|c ||c |c |c |c |c|}
\hline
Eigenfunctions& $i=1$ & $i=2$ & $i=3$ & $i=4$ & $i=5$\\ [0.5ex]
\hline\hline
$\varphi=\pi/16$&TRUE  & TRUE & FALSE $^{>99\%}$ & FALSE$^{>99\%}$  & FALSE$^{>99\%}$ \\
\hline
$\varphi=\pi/8$& TRUE & TRUE & FALSE$^{>99\%}$  & FALSE$^{>95\%}$  & FALSE$^{>99\%}$ \\
\hline
$\varphi=\pi/4$ & TRUE & TRUE &  FALSE$^{>99\%}$ & FALSE$^{>95\%}$ & FALSE$^{>99\%}$ \\
\hline
$\varphi=2\pi/5$ & TRUE & TRUE &  FALSE$^{>95\%}$ & TRUE& FALSE$^{>90\%}$ \\
\hline
\end{tabular}
\\[2ex]
\caption{\textit{Relevance of the differences of the first five eigenfunctions for different, relevant angles $\varphi$. Acceptance of the null hypothesis in \eqref{H0H1eigenvector} with $\Delta = \sqrt{2 - 2 \cos (\varphi)}$ is denoted by ``TRUE'' ($p$-value $>10\%$), rejection by ``FALSE''. For rejections, superindices indicate the probability of a less extreme event under the null.}\\[-7ex]}
\label{table:2}
\end{center}
\end{table}

To fully appreciate these results we have to take the estimated eigenvalues into account. The first five estimates before and after the  change point are given by
\begin{align*}
& \hat \tau^{(1)}_1= 38.38,\quad  \hat \tau^{(1)}_2= 0.56, \quad  \hat \tau^{(1)}_3= 0.26, \quad  \hat \tau^{(1)}_4= 0.24, \quad  \hat \tau^{(1)}_5= 0.21\\
 & \hat \tau^{(2)}_1= 41.15,\quad   \hat \tau^{(2)}_2= 0.45,\quad  \hat \tau^{(2)}_3= 0.40, \quad  \hat \tau^{(2)}_4= 0.34, \quad  \hat \tau^{(2)}_5= 0.29.
\end{align*}

The rapid decay of the eigenvalues indicates, that most of the data's variance -in fact about $90 \%$ for each sample- is explained by the first principal components. Due to the similarity of the first eigenfunctions
a low-dimensional representation of all data may be given, by projecting on a common, low-dimensional function space, a process much facilitating subsequent analysis. In contrast, due to great dissimilarities of the higher order eigenfunctions, finding a common space that captures, say $95\%$ of all variance will require much higher dimensions.

Bejond such fPCA-related considerations, the eigenvalues encode further, valuable information about the data. They indicate how strongly each eigenfunction contributes to the measured functions. Displaying the  eigenvalues $2-12$ (those still of larger order than $0.01$) in the below graphic reveals a striking trend: The estimated eigenvalues of the time period from $1894-1988$ are decaying faster than those of $1989-2018$. 

It should be noted that this trend persists if we base our estimates on a change point earlier in time, even if we split the data in equally sized halfs. Of course contamination of the second data  set then leads to smaller differences, but the underlying trend of slower decay of the eigenvalues from the earlier time period remains visible. This indicates that the observed effects are not simply due to a suggestive change point selection.

\begin{figure}[H]
\centering
\includegraphics[width=0.9\textwidth]{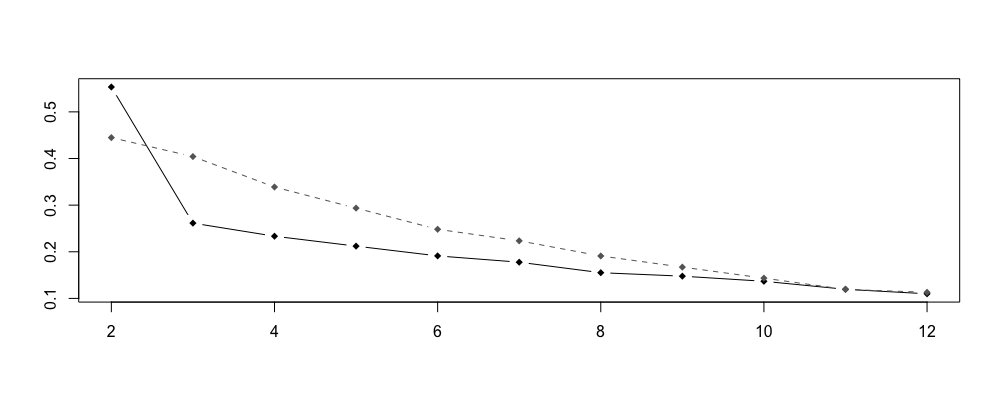}
\caption{\textit{Eigenvalues of order $2-12$ of the estimated covariance operators from 1884-1988 (solid) and 1989-2018 (dotted). }}
\end{figure}

To establish the relevance of these differences, we consider each eigenvalue with an individual threshold of relevance,
suited to its magnitude  (proportional to the size of $ \hat  \tau_j^{(1)}$) and apply the test \ref{test}.

\begin{table}[H] \centering
\begin{tabular}{|c|c|c|c|c|c|c|c|c|c|c|c|c|}
  \hline
  \diagbox[width=5em]{$\Delta_\tau$}{$j$}& 1$\,\, $ &2$\,\, $&3$\,\, $&4$\,\, $&5$\,\, $&6$\,\, $&7$\,\, $&8$\,\, $&9$\,\, $&10&11&12\\
  \hline
  $ \hat \tau_j^{(1)}/50$ & \cellcolor{black!75} & & \cellcolor{gray!25}& \cellcolor{black!75}& \cellcolor{black!75}& & & & & &  &\\
  \hline
    $ \hat \tau_j^{(1)}/100$ & \cellcolor{black!75} & & \cellcolor{gray!25}& \cellcolor{black!75}& \cellcolor{black!75}&\cellcolor{black!75} &\cellcolor{black!75} & & & &  &\\
  \hline
    $ \hat \tau_j^{(1)}/200$ & \cellcolor{black!75} & & \cellcolor{black!75}& \cellcolor{black!75}& \cellcolor{black!75}& \cellcolor{black!75}& \cellcolor{black!75}& \cellcolor{black!75}& \cellcolor{black!75}& &  &\\
  \hline
\end{tabular}
  \caption{\textit{Differences of the eigenvalues from the samples $1894-1988$ and $1989-2018$ for varying sizes of $\Delta_\tau$. Relevant differences with probability $\ge 95\%$ are marked in gray and $\ge 99\%$ in black. White cells suggest no rejection ($p$-value $>10\%$).}}
\end{table}

The visual inspection of the eigenvalues is consistent with the testing results. The eigenvalues of the covariance operators differ up to order $9$, with decreasing relevance.

One practical interpretation of these differences may be as follows: The faster decay of the eigenvalues of $\hat c^{(1)}$ compared to those of $\hat c^{(2)}$, indicates that the observations from $1894-1988$ are less influenced by higher order eigenfunctions, than those from $1989-2018$. Given that the eigenfunctions become more irregular with incresing order (compare Figure \ref{fig6}), this might imply rougher data, i.e. more short term variability temperatures recently, than in the first part of the $20$th century.

 \appendix

%%%%% Appendix A

 \section{Proofs} \label{appA}
 \def\theequation{A.\arabic{equation}}
\setcounter{equation}{0}

 For clarity of presentation we confine ourselves to the
case $j=1$, i.e. differences in the first eigenfunction and eigenvalue.
The general case follows by exactly the same arguments.

An important feature of the proofs is the replacement of the estimated change point $\hat \theta$, by the deterministic, true change point $\theta_0$. If we knew the true change point, we could construct the ideal, estimated covariance kernels
by $\tilde c^{(1)} (1,s,t)$ and $\tilde c^{(2)} (1,s,t)$ where for $\lambda \in [0,1]$
\begin{eqnarray}
\tilde{c}^{(1)}(  \lambda, s, t)&=& \frac{1}{\lfloor \lfloor N  \theta_0  \rfloor \lambda \rfloor} \sum_{n=1}^{\lfloor \lfloor N  \theta_0  \rfloor \lambda \rfloor}X_n(s)X_n(t), \label{optimalkernels1} \\ \nonumber
\tilde{c}^{(2)}(  \lambda, s, t)&=& \frac{1}{\lfloor (N- \lfloor N  \theta_0  \rfloor)  \lambda \rfloor} \sum_{n= N \theta_0  +1}^{\lfloor N  \theta_0  \rfloor+\lfloor (N-\lfloor N  \theta_0  \rfloor ) \lambda \rfloor} X_n(s)X_n(t). \label{optimalkernels2}
\end{eqnarray}
These kernels, as well as the corresponding eigensystems 
$$
\tilde v_{1, \lambda}^{(i)}, \tilde v_{2, \lambda}^{(i)},... \quad \mbox{ and } \quad 
 \tilde \tau_{1, \lambda}^{(i)} \ge  \tilde \tau_{2, \lambda}^{(i)} \ge ...
  $$
  for $i=1,2$, will be frequently referred to in the following section.

\subsection{Proof of Proposition \ref{proposition2a}}

Recall the   definition of the eigenvalues $\hat{\tau}_{1, \lambda}^{(i)}$ for $i=1,2$ of the
estimated  kernels $\hat c^{(i)}(\lambda, \cdot, \cdot)$ in \eqref{estimatedkernels1} and \eqref{estimatedkernels2}. In order to show the proposition, we
  establish the weak convergence
\begin{equation}
\{H_N(\lambda)\}_{\lambda \in [0,1]}:=\Big\{ \sqrt{N}\lambda^2 [(\hat{\tau}^{(1)}_{1, \lambda}- \hat{\tau}^{(2)}_{1,\lambda})^2 - (\tau^{(1)}_1 -\tau_1^{(2)})^2]\Big\} _{\lambda \in [0,1]}\to_D \{\lambda \sigma_E \mathbb{B}(\lambda)\Big\}_{\lambda \in [0,1]}, \label{eigenvalueconv}
\end{equation}
where $\mathbb{B}$ is a standard Brownian motion. The statement then follows by an application of the continuous mapping Theorem, as
$$\sqrt{N}(\hat E_{1,N}-E_{1}, \hat{V}_{1,N}) = \Big(H_N(1), \Big[\int_0^1  \{H_N(\lambda) - \lambda^2 H_N(1)\}^2 d\nu(\lambda)
\Big]^{1/2}\Big).$$
The proof of \eqref{eigenvalueconv} consists of two steps. \\[1ex]
\textbf{Step 1:}
First we demonstrate that using the change point estimate $\hat{\theta}$ is
asymptotically as good as knowing the true location $\theta_0$ of the change point, or more precisely
\begin{align}
\label{replacethetaev}
\sqrt{N}\lambda^2 \big[(\tilde{\tau}^{(1)}_{1, \lambda}- \tilde{\tau}^{(2)}_{1,\lambda})^2 - (\hat{\tau}^{(1)}_{1, \lambda}- \hat{\tau}^{(2)}_{1,\lambda})^2 \big]=o_P(1)
\end{align}
uniformly with respect to $\lambda \in [0,1]$. To establish this equality we show
\begin{eqnarray}
\lambda |\tilde{\tau}^{(1)}_{1, \lambda}- \hat{\tau}^{(1)}_{1, \lambda}| =o_P(1/\sqrt{N}) \label{replacethetaev2}
\end{eqnarray}
uniformly in $\lambda \in [0,1]$.  Deducing \eqref{replacethetaev}
from \eqref{replacethetaev2} is then a simple calculation.  To obtain an upper bound on the left side of \eqref{replacethetaev2}, we employ Lemma 2.2 from \cite{HK2012}, which yields

\begin{align}
&\lambda |\tilde{\tau}^{(1)}_{1, \lambda}- \hat{\tau}^{(1)}_{1, \lambda}| \le  \lambda \|\tilde{c}^{(1)}(\lambda,
\cdot, \cdot )-\hat{c}^{(1)}(\lambda,  \cdot, \cdot )\|.
\label{upperboundeve}
\end{align}
For the difference of the kernels we obtain
\begin{align}
&\lambda\big[\tilde{c}^{(1)}(\lambda,  s, t )-\hat{c}^{(1)}(\lambda,  s, t
)\big] \label{termsreaarangede}\\
=&\frac{1}{ \lfloor N  \theta_0  \rfloor}\sum_{n=1}^{\lfloor \lfloor N
\theta_0\rfloor\lambda \rfloor} X_n(t)X_n(s) -\frac{1}{N\hat{\theta} } \sum_{n=1}
^{\lfloor N \hat{\theta}\lambda \rfloor} X_n(t)X_n(s) +\mathcal{O}
_P(N^{-1})\nonumber\\
=& \frac{sign(\theta_0-\hat{\theta})}{\lfloor N (\theta_0 \lor
\hat{\theta}) \rfloor} \sum_{n=\lfloor \lfloor(  \hat{\theta} \land   \theta_0
)N\rfloor \lambda\rfloor}^{\lfloor \lfloor (  \hat{\theta} \lor
\theta_0 )N \rfloor \lambda \rfloor } X_n(t)X_n(s)\nonumber+ \Big( \frac{1}{\lfloor N  \theta_0  \rfloor}-\frac{1}{N\hat{\theta}} \Big) \sum_{n=1}^{\lfloor \lfloor( \hat{\theta} \land
\theta_0)N\rfloor\lambda\rfloor } X_n(t)X_n(s)+\mathcal{O}_P(N^{-1}), \nonumber
\end{align}
where the second equality follows by a straightforward rearrangement of the terms.
Notice that $|1/\hat{\theta}-N/\lfloor N  \theta_0  \rfloor|=o_P(1/\sqrt{N})$, which follows immediately by Proposition \ref{Proposition2}.  An application of the triangle inequality
shows that the $L^2([0,1]\times[0,1])$-norm of the second term on right side of
\eqref{termsreaarangede}  can be
  bounded by
$$\Big|\frac{N}{\lfloor N  \theta_0  \rfloor}-\frac{1}{\hat{\theta}}\Big|  \frac{1}{N}\sum_{n=1}^{N}\|
X_n \otimes X_n  \| =o_P(1/\sqrt{N}),$$
where we have applied Birkhoff's Theorem.
The $L^2([0,1]\times[0,1])$-norm of the first term on the right of
\eqref{termsreaarangede} is also of order $o_P(1/\sqrt{N}).$ Counting the
summands we see that centering  only yields a further term of order $o_P(1/
\sqrt{N})$. We now use Theorem \ref{thmapp1} from Appendix \ref{appB} to get
\begin{align}
& \sup_\lambda \Big\{ \int_0^1\int_0^1\Big[ \frac{sign(\theta_0-\hat{\theta})}
{\lfloor N (\theta_0 \land   \hat{\theta}) \rfloor} \sum_{n=\lfloor \lfloor (  \hat{\theta} \land
\theta_0 )N\rfloor\lambda\rfloor}^{\lfloor\lfloor  (  \hat{\theta} \lor  \theta_0 )N\rfloor
\lambda \rfloor } \big[ X_n(t)X_n(s)-\mathbb{E}[X_n(t)X_n(s)] \big] \Big]^2 ds dt \Big\}^{1/2}
\label{Prop2eq2e} \\
=& \mathcal{O}_P\Big( \frac{1}{\sqrt{N}}\sup_\lambda \Big\{
\int_0^1\int_0^1\Big[\Gamma_N(\lambda\lfloor (  \hat{\theta} \lor  \theta_0 ) N\rfloor
/N,s,t )-\Gamma_N(\lambda\lfloor (  \hat{\theta} \land  \theta_0 ) N\rfloor /N,s,t)
\Big]^2 ds dt\Big\}^{1/2} \Big), \nonumber
\end{align}
where $\Gamma_N$ is a Gaussian process   defined by
$$ \Gamma_N(x,s,t):=\sum_{l\ge 1} \Lambda_{l} \mathbb{B}_{l,N}(x)\Psi_{l}
(s,t),$$
 $\{ \mathbb{B}_{l,N}: l \in \mathbb{N}\}$ are independent Brownian motions, $\{\Lambda_{l}:
l \in \mathbb{N}\}$ are non-negative, square summable values  and $\{\Psi_{l}: l \in \mathbb{N}\}$
is an orthonormal basis of $L^2([0,1]\times[0,1])
$. Enlarging the set over which we
maximize, we see that the supremum in \eqref{Prop2eq2e}  is   bounded by
\begin{align*}
\sup_{|\lambda-\mu|\le |\hat{\theta}-\theta_0|}\Big( \int_0^1\int_0^1\big[
\Gamma_N( \lambda,s,t) -\Gamma_N(\mu,s,t) \big]^2 ds dt\Big)^{1/2} ,
\end{align*}
and   the definition of $\Gamma_N$ yields
$$ \sup_{|\lambda-\mu|\le |\hat{\theta}-\theta_0|}\Big( \int_0^1\int_0^1\big[
\Gamma_N( \lambda,s,t) -\Gamma_N(\mu,s,t) \big]^2 ds dt\Big)^{1/2}=\sup_{|
\lambda-\mu|\le |\hat{\theta}-\theta_0|}\Big( \sum_{l\ge 1} \Lambda_{l}^2
\big[\mathbb{B}_{l,N}(\lambda)-\mathbb{B}_{l,N}(\mu)\big]^2\Big)^{1/2}.$$
 Its expectation is  bounded by
$$\Big( \mathbb{E} \Big[ \sup_{|\lambda-\mu|\le | \hat{\theta}-\theta_0|}(\mathbb{B}
(\lambda)-\mathbb{B}(\mu))^2 \Big ] \sum_{l\ge 1} \Lambda_{l}^2 \Big)^{1/2}  ,$$
where $\mathbb{B}$ is a standard Brownian motion. The expectation converges to
$0$ by application of the dominated convergence Theorem together with the almost sure continuity of the paths of the Brownian motion (see Billingsley, Section 37). These considerations yield \eqref{replacethetaev2} and hence \eqref{replacethetaev}. \\[1ex]

\textbf{Step 2:} We now prove the weak convergence \eqref{eigenvalueconv}. Step 1 and straightforward calculations yield 
\begin{eqnarray*}
 \sqrt{N}\lambda^2[ (\tilde{\tau}^{(1)}_{1,\lambda} - \tilde{\tau}^{(2)}_{1,
\lambda})^2- (\tau^{(1)}_1  +\tau_1^{(2)})^2]
 &=&  \sqrt{N}\lambda^2
[\tilde{\tau}^{(1)}_{1,\lambda} - \tilde{\tau}^{(2)}_{1,\lambda} -
\tau^{(1)}_1  -\tau_1^{(2)}]^2  \\
&& + 2\sqrt{N}\lambda^2 (\tau^{(1)}_1  -\tau_1^{(2)})(\tilde{\tau}^{(1)}_{1,\lambda} - \tilde{\tau}^{(2)}_{1,\lambda}-\tau^{(1)}_1  +\tau_1^{(2)} ) \\
  &=&2\sqrt{N}\lambda^2 (\tau^{(1)}_1  -\tau_1^{(2)})(\tilde{\tau}^{(1)}_{1,\lambda} - \tilde{\tau}^{(2)}_{1,\lambda}-\tau^{(1)}_1  +\tau_1^{(2)} ) +o_P(1).
\end{eqnarray*}

  For further analysis of the quantity $G_N(\lambda) = \sqrt{N}\lambda^2 (\tilde{\tau}^{(1)}_{1,
  \lambda} - \tilde{\tau}^{(2)}_{1,\lambda}-\tau^{(1)}_1  +\tau_1^{(2)}
  )$ we use Proposition
  \ref{propositionev}, in Appendix \ref{appB}, which gives
 \begin{align*}
&G_N(\lambda)=\sqrt{N}\lambda^2 (\tilde{\tau}^{(1)}_{1,\lambda} - \tilde{\tau}^{(2)}_{1,\lambda} - \tau^{(1)}_1  +\tau_1^{(2)})
=R_{N}^{(1)} (\lambda)  + R_{N}^{(2)}  (\lambda) +o_P(1) ~,
\end{align*}
where
 \begin{align*}
R_{N}^{(1)} (\lambda) =&\frac{\sqrt{N} \lambda}{\lfloor N \theta_0 \rfloor } \sum_{n=1}^{\lfloor \lfloor N \theta_0 \rfloor \lambda \rfloor } \int_0^1\int_0^1(X_n^{(1)}(s) X_n^{(1)}(t)-c^{(1)}(s,t)) v_1^{(1)}(s) v_1^{(1)}(t) ds dt \\
R_{N}^{(2)} (\lambda) =&\frac{\sqrt{N}\lambda}{\lfloor N (1-\theta_0) \rfloor } \sum_{n=\lfloor N \theta_0 \rfloor +1}^{\lfloor \lfloor N (1-\theta_0) \rfloor \lambda \rfloor +\lfloor N \theta_0 \rfloor } \int_0^1\int_0^1(X_n^{(2)}(s) X_n^{(2)}(t)-c^{(2)}(s,t)) v_1^{(2)}(s) v_1^{(2)}(t) ds dt.
\end{align*}
Weak convergence of the process $\{G_N(\lambda)\}_{\lambda \in [0,1]}$ now  follows by an application of the continuous mapping Theorem and the weak convergence of the vector valued process
 \begin{align} \label{tupleprocess}
&\big( R_{N}^{(1)} (\lambda) , R_{N}^{(2)} (\lambda)  \big  )_{\lambda \in [0,1]}.
\end{align}
For a proof of this statement we show
asymptotic tightness and convergence of the finite dimensional distributions. We  therefore introduce the random variables
\begin{equation}
Y_n^{(i)}:=\int_0^1\int_0^1 \big[ X_n^{(i)}(s) X_n^{(1)}(t)-c^{(i)}(s,t) \big] v_1^{(i)}(s) v_1^{(i)}(t) ds dt \quad i=1,2 \label{Y0}
\end{equation}
   and
\begin{equation}
Y_{n,m}^{(i)}:=\int_0^1\int_0^1\big[ X_{n,m}^{(i)}(s) X_{n,m}^{(i)}(t)-c^{(i)}(s,t) \big] v_1^{(i)}(s) v_1^{(i)}(t) ds dt \quad i=1,2,\label{Y0m}
\end{equation}
where the random functions $X^{(i)}_{n,m}$ are defined in \eqref{m-dependentRV}.
Asymptotic tightness can be shown coordinate-wise, such that we verify it
exemplarily for the first component. This can be rewritten as the process
\begin{equation}
\big(R^{(1)}_N(\lambda)\big)_{\lambda \in [0,1]}= \Big(\frac{\sqrt{N} \lambda}{\lfloor N \theta_0 \rfloor } \sum_{n=1}^{\lfloor \lfloor N
\theta_0 \rfloor \lambda \rfloor } Y_n^{(1)} \Big)_{\lambda \in [0,1]}.
\label{firstcomponent}
\end{equation}
As tightness of a stochastic process $
(G(\lambda))_{\lambda \in [0,1]}$  in $\ell^\infty[0,1]$ implies tightness of $(\lambda
G(\lambda))_{\lambda \in [0,1]}$, it will suffice to show tightness of
$(R^{(1)}_N (\lambda)/\lambda)_{\lambda \in [0,1]}$. To prove this assertion we note
\begin{eqnarray} \label{tightness}
&& \lim_{\delta \to 0} \lim_{N \to \infty} \mathbb{P}\Big( \sup_{|r-q|<\delta}
\frac{\sqrt{N}}{\lfloor N \theta_0 \rfloor }\Big|  \sum_{n=1}^{\lfloor \lfloor N \theta_0
\rfloor q \rfloor } Y_n^{(1)}-\sum_{n=1}^{\lfloor \lfloor N \theta_0 \rfloor r \rfloor }
Y_n^{(1)}\Big|>\varepsilon\Big)  \\
&& \leq \lim_{\delta \to 0} \lim_{N \to \infty} \mathbb{P}\Big( \sup_{|r-q|<\delta} \frac{\sqrt{N}}{\lfloor N \theta_0 \rfloor }\Big|  \sum_{n=1}^{\lfloor \lfloor N \theta_0 \rfloor q \rfloor } Y^{(1)}_{n,m}-\sum_{n=1}^{\lfloor \lfloor N \theta_0 \rfloor r \rfloor } Y^{(1)}_{n,m}\Big|>\varepsilon/2\Big) \nonumber\\
&& + 2\lim_{N\to \infty}  \mathbb{P}\Big( \sup_{q\in [0,1]} \frac{\sqrt{N}}{\lfloor N \theta_0 \rfloor }\Big|  \sum_{n=1}^{\lfloor \lfloor N \theta_0 \rfloor q \rfloor } Y^{(1)}_{n}- Y^{(1)}_{n,m}\Big|>\varepsilon/4\Big). \nonumber
\end{eqnarray}

Adapting
the proof of Lemma 2.1 in \cite{Berkes2013} shows that under the
assumptions stated in Section \ref{sec2}
\begin{align}
\label{berkeslemma21} \lim_{m \to \infty} \limsup_{N \to \infty} \mathbb{P}\Big(
\max_{k=1,...,N} \Big\| &\frac{1}{\sqrt{N}} \sum_{n=1}^k (X_n^{(i)} \otimes X_n^{(i)}
-  X_{n,m}^{(i)} \otimes X_{n,m} ^{(i)})\Big\|>x\Big)=0
\end{align}
for any $x>0$.
By this result and the Cauchy-Schwarz inequality, the second term on the right side of \eqref{tightness} converges to $0$ as $m \to \infty$.

 For the first term we define for $h=1,...,m$
$$ Q(h, T):=\{n \le T: n \,\textnormal{mod} \,m=h \}  $$
  and split the sample  $Y_{1,m}^{(1)},...,Y_{\lfloor N \theta_0\rfloor,m}^{(1)}$ in $m$ groups of independent identically distributed random variables, which gives
\begin{align*}
&\mathbb{P}\Big( \sup_{|r-q|<\delta} \frac{\sqrt{N}}{\lfloor N \theta_0 \rfloor }\Big|  \sum_{n=1}^{\lfloor \lfloor N \theta_0 \rfloor q \rfloor } Y^{(1)}_{n,m}-\sum_{n=1}^{\lfloor \lfloor N \theta_0 \rfloor r \rfloor } Y^{(1)}_{n,m}\Big|>\varepsilon/2\Big)\\
=& \mathbb{P}\Big( \sup_{|r-q|<\delta} \frac{\sqrt{N}}{\lfloor N \theta_0 \rfloor }\Big|  \sum_{h=1}^m \sum_{n \in Q(h, \lfloor \lfloor N \theta_0 \rfloor q \rfloor)} Y^{(1)}_{n,m}-\sum_{n \in  Q(h, {\lfloor \lfloor N \theta_0 \rfloor r \rfloor })} Y^{(1)}_{n,m}\Big|>\varepsilon/2\Big)\\
 \le & \mathbb{P}\Big( \sup_{|r-q|<\delta} m \max_{1 \le h \le m} \frac{\sqrt{N}}{\lfloor N \theta_0 \rfloor }\Big|  \sum_{n \in Q(h, \lfloor \lfloor N \theta_0 \rfloor q \rfloor)} Y^{(1)}_{n,m}-\sum_{n \in  Q(h, {\lfloor \lfloor N \theta_0 \rfloor r \rfloor })} Y^{(1)}_{n,m}\Big|>\varepsilon/2\Big)\\
\le & m\mathbb{P}\Big( \sup_{|r-q|<\delta} \frac{\sqrt{N}}{\lfloor N \theta_0 \rfloor }\Big|  \sum_{n \in Q(1, \lfloor \lfloor N \theta_0 \rfloor q \rfloor)} Y^{(1)}_{n,m}-\sum_{n \in  Q(1, {\lfloor \lfloor N \theta_0 \rfloor r \rfloor })} Y^{(1)}_{n,m}\Big|>\varepsilon/(2 m)\Big).
\end{align*}
In the last step we have used that all sums are identically distributed and we have assumed  without loss of generality
that  $\lfloor \lfloor N \theta_0\rfloor q\rfloor$ and $\lfloor \lfloor N \theta_0\rfloor r\rfloor$  are divisible by $m$ (otherwise the remaining  term,  is  asymptotically negligible). Taking the limit with respect to $N$ we observe that the right side converges to $0$, which follows due to asymptotic tightness of partial sums of independent random variables, as presented in \cite{VW1996}. \\

The remaining part of this proof consists in verifying the marginal convergence of \eqref{tupleprocess}. More precisely, we prove by an application of the Cramer-Wold device 
 for a finite number of parameters $0 \le \lambda_1\le ... \le \lambda_K \le 1$ weak convergence of the random vector 
$$
[R_N^{(1)}(\lambda_1), R_N^{(2)}(\lambda_1),...,R_N^{(1)}(\lambda_K), R_N^{(2)}(\lambda_K)].
$$
Again   \eqref{berkeslemma21} and the Cauchy-Schwarz inequality show that we can replace this vector by its $m$-dependent version $[R_{N,m}^{(1)}(\lambda_1), R_{N,m}^{(2)}(\lambda_1),...,R_{N,m}^{(1)}(\lambda_K), R_{N,m}^{(2)}(\lambda_K)]$, where
$$ \big(R^{(1)}_{N,m}(\lambda),  R^{(2)}_{N,m}(\lambda)\big)=\Big(\frac{\sqrt{N} \lambda}{\lfloor N \theta_0 \rfloor } \sum_{n=1}^{\lfloor \lfloor N
\theta_0 \rfloor \lambda \rfloor } Y_{n,m}^{(1)} , \frac{\sqrt{N}\lambda}{\lfloor N (1-\theta_0) \rfloor } \sum_{n=\lfloor N \theta_0 \rfloor +1}^{\lfloor \lfloor N (1-\theta_0) \rfloor \lambda \rfloor +\lfloor N \theta_0 \rfloor } Y_{n,m}^{(2)}\Big)_{\lambda \in [0,1]}.$$
  Now let $a_1,...,a_K$ and $b_1,...,b_K$ be arbitrary real numbers and consider the sum

 \begin{align}
\sum^K_{j=1} a_j R^{(1)}_{N,m} (\lambda_j) + b_j R^{(2)}_{N,m} (\lambda_j) = &\sqrt{N}\sum_{j=1}^K \Big\{ \sum_{n=\lfloor \lambda_{j-1}\lfloor N
\theta_0\rfloor \rfloor}^{\lfloor \lambda_j\lfloor N \theta_0\rfloor \rfloor}Y_{n,m}^{(1)}
\frac{\lambda_j}{\lfloor N \theta_0 \rfloor} \sum_{l=1}^j a_l \Big\} \label{clttermss}\\
+&\sqrt{N}\sum_{j=1}^K \Big\{ \sum_{n=\lfloor \lambda_{j-1}N\ \rfloor +
\lfloor N \theta_0\rfloor}^{\lfloor \lambda_{j}N\ \rfloor +\lfloor N \theta_0\rfloor}
Y_{n,m}^{(2)} \frac{\lambda_j}{\lfloor N (1-\theta_0) \rfloor}  \sum_{l=1}^j b_l\Big\}
\nonumber
\end{align}
(here we put $\lambda_0 = 0)$. To establish weak convergence we use the central limit theorem from \cite{Be1973} for $m
$-dependent triangular arrays of random variables. The only non-trivial assumption of this theorem in our case is the convergence of the variance of \eqref{clttermss}, which will be established next. As we can see the covariance of the two groups for $i=1, 2$ converges to $0$, since
$$\frac{1}{N}\mathbb{C}ov\bigg( \sum_{j=1}^K  \sum_{n=
\lambda_{j-1}\lfloor N \theta_0\rfloor}^{\lfloor \lambda_j
\lfloor N \theta_0\rfloor \rfloor}Y_{n,m}^{(1)} \frac{N \lambda_j}{\lfloor N
\theta_0 \rfloor}  \sum_{l=1}^j a_l,  \sum_{j=1}^K  \sum_{n=\lfloor \lambda_{j-1}N\
\rfloor +\lfloor N \theta_0\rfloor}^{\lfloor \lambda_{j}N\ \rfloor +\lfloor N
\theta_0\rfloor}Y_{n,m}^{(2)} \frac{N \lambda_j}{\lfloor N (1-\theta_0) \rfloor}
\sum_{l=1}^j b_l\bigg)  $$
 is of order $o(1)$ due to $m$-dependence. Consequently, we can investigate the variance of each of the two terms in
\eqref{clttermss} separately. Iterating the argument yields that we may confine ourselves to the variance of  the terms
$$ \sqrt{N}\sum_{n=\lambda_{j-1}\lfloor N \theta_0\rfloor}^{\lfloor \lambda_j\lfloor
N \theta_0\rfloor \rfloor}Y_{n,m}^{(1)} \frac{\lambda_j}{\lfloor N \theta_0
\rfloor}  \sum_{l=1}^j a_l \quad \textnormal{and} \quad \sqrt{N} \sum_{n=\lfloor
\lambda_{j'-1}N\ \rfloor +\lfloor N \theta_0\rfloor}^{\lfloor \lambda_{j'}N\ \rfloor +
\lfloor N \theta_0\rfloor}Y_{n,m}^{(2)} \frac{\lambda_{j'}}{\lfloor N (1-\theta_0)
\rfloor}  \sum_{l=1}^{j'} b_l,$$
 for $j, j'=1,...,K$ separately.
 The subsequent convergence arguments are now the same for all remaining terms and we
 only  exemplarily  consider the first one and obtain 
\begin{align}
&\mathbb{V}ar \Big\{\sqrt{N}\sum_{n=1}^{\lfloor \lambda_1\lfloor N \theta_0\rfloor
\rfloor}Y_{n,m}^{(1)} \frac{\lambda_1}{\lfloor N \theta_0 \rfloor} a_1  \Big\}=
\frac{\lambda_1^2 N a_1^2}{\lfloor N \theta_0 \rfloor^2} \mathbb{V}ar \Big\{
\sum_{n=1}^{\lfloor \lambda_1\lfloor N \theta_0\rfloor \rfloor}Y_{n,m}^{(1)} \Big
\} \label{longrunconvergence}  \\
=& \frac{\lambda_1^2\lfloor \lambda_1\lfloor N \theta_0\rfloor \rfloor N a_1^2}{\lfloor
N \theta_0 \rfloor^2} \sum_{|k| \le m}(1-|k|/\lfloor \lambda_1\lfloor N \theta_0\rfloor \rfloor)\mathbb{C}ov(Y_{0,m}^{(1)},
Y^{(1)}_{k,m}) \to  \lambda_1^3 \frac{a_1^2}{\theta_0} \sum_{|k| \le m} \mathbb{C}ov(Y_{0,m}^{(1)},
Y^{(1)}_{k,m}).  \nonumber
\end{align}
Finally we have to show that for $m \to \infty$ the  sum on the right-hand side converges to
\begin{align}
(\sigma^{(1)}_E)^2= \mathbb{V}ar(Y_0^{(1)})+2\sum_{l \ge 1} \mathbb{C}ov
(Y_{0}^{(1)},Y_{l}^{(1)})>0,\label{longrunvariance1e}
\end{align}
which follows from
\begin{equation}\label{CN}
C_m:=\sum_{|k|\le m} \big[\mathbb{C}ov(Y_0^{(1)}, Y_k^{(1)})-\mathbb{C}ov(Y_{0,m}^{(1)},Y^{(1)}_{k,m})\big] \to 0.
\end{equation}
Note that $\sigma^{(1)}_E$ is positive by assumption.
To prove \eqref{CN} consider the estimate
\begin{align*}
C_m
%\le&  \sum_{|k|\le m} \big|\mathbb{C}ov(Y_0^{(1)}, Y_k^{(1)})-
%\mathbb{C}ov(Y_{0,m}^{(1)},Y^{(1)}_{k,m})\big| \\
\le & \sum_{|k|\le m} \big|\mathbb{C}ov(Y_0^{(1)}, Y_k^{(1)}-Y^{(1)}_{k,m})\big|+\sum_{|k|\le m}\big|\mathbb{C}ov(Y_{0,m}^{(1)}-Y_0^{(1)},Y^{(1)}_{k,m})\big|.
\end{align*}
Each term can be   bounded according to Cauchy-Schwarz by
\begin{equation}
\sum_{|k|\le m} \mathbb{E}[(Y_{0,m}^{(1)}-Y_0^{(1)})^2]^{1/2}
\mathbb{E}[(Y^{(1)}_0)^2]^{1/2} =  m \mathbb{E}[(Y^{(1)}_0)^2]^{1/2}
\mathbb{E}[(Y_{0,m}^{(1)}-Y_0^{(1)})^2]^{1/2}.
\label{summabilargument}
\end{equation}
The expectation on the right can be
further analyzed plugging in the definitions of $Y_0$ and $Y_{0,m}$ (see \eqref{Y0}, \eqref{Y0m}  ).
\begin{align*}
\mathbb{E}\big[(Y_{0,m}^{(1)}-Y_0^{(1)})^2\big]^{1/2}=&
\mathbb{E} \Big[ \Big\{\int_0^1 \int_0^1 [X_{0,m}^{(1)}(s)X_{0,m}^{(1)}(t)-X_{0}^{(1)}(s)X_{0}^{(1)}(t)]v_1^{(1)}(s) v_1^{(1)}(t)ds dt \Big\}^2  \Big]^{1/2} \\
\le &  \mathbb{E}\big[ \|X_0\otimes X_0-X_{0,m}\otimes
X_{0,m}\|^2 \big]^{1/2}\\
\le & 2 \| \mathbb{E}\big[ \|X_0-X_{0,m}\|^2 \|X_0\|^2 + \|X_0-X_{0,m}\|^2 \|X_{0,m}\|^2 \big]^{1/2}\\
\le & 4 \mathbb{E}\big[ \|X_0-X_{0,m}\|^4\big]^{1/4}
\mathbb{E}\big[ \|X_0\|^4\big]^{1/4}.
\end{align*}
Using this estimate in \eqref{summabilargument} yields the upper bound
\begin{equation}
C_m \le  L (2m+1) \mathbb{E}\big[ \|X_0-X_{0,m}\|^4\big]^{1/4},
\label{CNtozero}
\end{equation}
for some fixed constant $L>0$. According to Assumption \ref{assumption1} there exists a sufficiently small number $\eta>0$, such that the sequence $\mathbb{E}\big[ \|X_0-X_{0,m}\|^4\big]^{1/4+\eta}$ is summable, and we obtain
$$
\lim_{m \to \infty}  (2m+1) \mathbb{E}\big[ \|X_0-X_{0,m}\|^4\big]^{1/4} = 0.
$$
Analogously to $\sigma^{(1)}_E$ we define the long run variance for the random variables $Y_n^{(2)}$ as
  \begin{equation}
(\sigma^{(2)}_E)^2= \mathbb{V}ar(Y_0^{(2)})+2\sum_{l \ge 1} \mathbb{C}ov(Y_{0}^{(2)},Y_{l}^{(2)})>0 \label{longrunvariance2e}
\end{equation}
(recall that   $\sigma^{(2)}_E$ is positive by assumption). We have now  verified all conditions of the Theorem in \cite{Be1973} and this implies convergence of the finite dimensional distributions. Consequently it follows that
\begin{align*}
&\Big( \frac{\sqrt{N} \lambda}{\lfloor N \theta_0 \rfloor} \sum_{n=1}^{\lfloor \lfloor N \theta_0 \rfloor \lambda \rfloor}
Y_n^{(1)} , \frac{\sqrt{N} \lambda}{\lfloor N (1-\theta_0) \rfloor} \sum_{n= \lfloor N \theta _0\rfloor +1}^{ \lfloor N \theta_0 \rfloor +\lfloor \lfloor N (1-\theta_0) \rfloor \lambda \rfloor} Y_n^{(2)} \Big)_{ \lambda \in [0,1]} \\\to_D& \left(    \sigma_E^{(1)}\lambda\mathbb{B}_1(\lambda)/\sqrt{\theta_0}, \sigma_E^{(2)} \lambda\mathbb{B}_2(\lambda)/\sqrt{1-\theta_0} \right)_{ \lambda \in [0,1]},
\end{align*}
where $\mathbb{B}_1$ and $\mathbb{B}_2$ are independent, standard Brownian motions.
Combining these considerations with the continuous mapping theorem shows  that
\begin{align*}
&\Big\{ \sqrt{N}\lambda^2 [(\tilde{\tau}^{(1)}_{1, \lambda}- \tilde{\tau}^{(2)}_{1,\lambda})^2 - (\tau^{(1)}_1-\tau_1^{(2)})^2]\Big\} _{\lambda \in [0,1]} \\
 \to_D & \left( 2|\tau_1^{(1)}-\tau_1^{(2)}|\lambda [\sigma_E^{(1)}/\sqrt{\theta_0} \mathbb{B}_1(\lambda)+\sigma_E^{(2)}/\sqrt{1-\theta_0}\mathbb{B}_2(\lambda)]\right)_{\lambda \in [0,1]},
\end{align*}
which completes the proof. \hfill $\Box$

\subsection{Proof of Theorem \ref{theorem1}}

Recall that the probability of rejecting the null hypothesis in \eqref{H0H1eigenvector} is given by
$$\mathbb{P}\Big( \frac{\hat E_{1,N} -\Delta_\tau}{\hat{V}_{1,N}}>q_{1-\alpha}\Big)= \mathbb{P}\Big( \frac{\hat E_{1,N} -E_{1}}{\hat{V}_{1,N}}>q_{1-\alpha}+\frac{\Delta_\tau -E_1}{\hat{V}_{1,N}}\Big) .$$ Suppose that  $E_1>0$ i.e. that the  first eigenvalues of $c^{(1)}$ and $c^{(2)}$ are different. Then by  Proposition \ref{Proposition2},
\begin{equation}\label{H2e}
\frac{\hat E_{1,N}-E_{1}}{\hat{V}_{1,N}} \to_D W,
\end{equation}
where the random variable $W$ is defined in \eqref{W}. The same result shows that
$\hat{V}_{1,N}$ is of order $o_P(1)$ which implies that $(\Delta_\tau -E_1)/\hat{V}_{1,N}$ converges to $+\infty$ if $\Delta_\tau >E_1$ and to $-\infty$ if $\Delta_\tau <E_1$, both in probability. This implies consistency and level $\alpha$ as stated in \eqref{levelalpha}, in the case $E_1 > 0$.

Now suppose that $E_{1}=0$. In this case we show that $\hat E_{1,N}$ and
$\hat{V}_{1,N}$ are of order $o_P(1)$. Then, with probability converging to $1$, $\hat E_{1,N}/\hat{V}_{1,N}$
is asymptotically negligible compared to $\Delta_\tau/\hat{V}_{1,N}$ which converges to infinity.  First,
suppose that the kernels $c^{(1)}$ and $c^{(2)}$ are not equal.  Proposition
\ref{Proposition2} implies that we may replace the change point estimator in
numerator $\hat E_{1,N}$ and denominator $\hat{V}_{1,N}$ by the actual change point, only
incurring asymptotically vanishing terms.  More precisely, the denominator in \eqref{H2e} equals
\begin{eqnarray*}
\hat{V}_{1,N} &=& \Big( \int_0^1 \lambda^4 \Big[ (\tilde{\tau}_{1, \lambda}^{(1)}
-\tilde{\tau}_{1, \lambda}^{(2)})^2 - (\tilde{\tau}_{1, 1}^{(1)}-\tilde{\tau}
_{1, 1}^{(2)})^2 \Big]^2 \nu(d \lambda) \Big)^{1/2}+o_P(1)\\
 &=& \Big( \int_0^1 \lambda^4 \Big[  (\tilde{\tau}_{1, \lambda}^{(1)}-\tau^{(1)}_1+\tau^{(2)}_1-\tilde{\tau}_{1, \lambda}^{(2)})^2  -  (\tilde{\tau}_{1, 1}^{(1)}-\tau^{(1)}_1+\tau^{(2)}_1-\tilde{\tau}_{1, 1}^{(2)})^2 \Big]^2 \nu(d \lambda) \Big)^{1/2}+o_P(1),
\end{eqnarray*}
where we have used the equality of the eigenvalues.
Now  Proposition  \ref{propositionev} in Appendix \ref{appB} shows that    $\hat{V}_{1,N}=o_P(1)$. Finally,   similar, but simpler arguments show that $\hat E_{1,N}=o_P(1)$.

In the case of equality $c^{(1)} =c^{(2)}$ the estimator of the change point assumes some uninformative
value inside the interval  $[\varepsilon, 1-\varepsilon]$ and we obtain
\begin{align*}
\hat{V}_{1,N}^2\le & 4  \sup_\lambda \Big(\lambda^2 (\hat{\tau}_{1, \lambda}^{(1)}-\tau^{(1)}_1+\tau^{(2)}_1-\hat{\tau}_{1, \lambda}^{(2)})^2 \Big)^2 \\
\le & 64  \sup_\lambda\Big(\lambda^2 (\hat{\tau}_{1, \lambda}^{(1)}-\tau^{(1)}_1)^2 \Big)^2
+ 64   \sup_\lambda\Big(\lambda^2 (\tau^{(2)}_1-\hat{\tau}_{1, \lambda}^{(2)})^2 \Big)^2 .
\end{align*}
Consider now the inequality
$$\Big(\lambda^2 (\hat{\tau}_{1, \lambda}^{(1)}-\tau^{(1)}_1)^2 \Big)^2 \le \sup_{\theta \in [\varepsilon, 1 - \varepsilon]}  \Big\| \frac{1}{ \lfloor N \theta \rfloor}\sum_{n=1}^{\lfloor \lfloor N \theta \rfloor \lambda\rfloor} \big( X_n \otimes X_n-c^{(1)} \big) \Big\|^4+o_P(1)$$
which follows from Lemma 2.2 in \cite{HK2012}. The above expression is of order $o_P(1)$
uniformly in $\lambda$ by Lemma B.1 in the supplementary material
of \cite{ARS2018}.  Applying the same argument to the second term yields
that   $\hat{V}_{1,N}=o_P(1)$. The corresponding arguments for the estimate $\hat E_{1,N}=o_P(1)$ are similar and therefore again omitted. \hfill $\Box$

\subsection{Proof of Theorem \ref{theorem2}} \label{proof_thm_2}

Recall the   definition of the eigenfunctions $\hat{v}_{1,\lambda}^{(i)}$  of the
estimated  kernels $\hat c^{(i)}(\lambda, \cdot, \cdot)$ defined in  \eqref{estimatedkernels1} and \eqref{estimatedkernels2}. Similarly as for the proof of Theorem \ref{theorem1} we prove the weak convergence
\begin{equation}
\{H_N(\lambda)\}_{\lambda \in [0,1]}  \to_D  \{\lambda \sigma_D \mathbb{B}(\lambda)\}
_{\lambda \in [0,1]}, \label{weakconvProp2}
\end{equation}
where the process $\{H_N (\lambda) \}_{\lambda \in [0,1]}$ is defined by
\begin{align}
&H_N(\lambda):= \sqrt{N}\Big[ \int_0^1\left(\hat{v}^{(1)}_{1, \lambda}(t)-
\hat{v}^{(2)}_{1, \lambda}(t)\right)^2 \lambda^2 dt - \int_0^1\left(
v^{(1)}_1(t)-v^{(2)}_1(t)\right)^2 \lambda^2 dt \Big]
\end{align}
The result then follows by similar arguments as given in the proof of Theorem \ref{theorem1} and the continuous mapping theorem, which implies the weak convergence of the tuple
$$
\big ( \sqrt{N} (\hat D_{1,N} - D_1), \sqrt{N} \hat U_{1,N}\big)=\Big(H_N(1), \Big[\int_0^1  \{H_N(\lambda) - \lambda^2 H_N(1)\}^2 d\nu(\lambda)
\Big]^{1/2}\Big).
$$
First we replace the estimated by the true change point showing
\begin{align}
\label{replacetheta}
\sqrt{N}\left[ \int_0^1\left(\tilde{v}^{(1)}_{1, \lambda}(t)-
\tilde{v}^{(2)}_{1, \lambda}(t)\right)^2 \lambda^2 dt - \int_0^1 \left(\hat{v}^{(1)}_{1, \lambda}(t)-
\hat{v}^{(2)}_{1, \lambda}(t)\right)^2 \lambda^2 dt \right]=o_P(1)
\end{align}

uniformly in $\lambda \in [0,1]$. To establish this equality we prove
\begin{eqnarray}
\lambda^2 \|\tilde{v}^{(1)}_{1, \lambda}- \hat{v}^{(1)}_{1, \lambda}\| =o_P(1/\sqrt{N}) \label{ProofPro2.4eq1}
\end{eqnarray}
uniformly in $\lambda \in [0,1]$.  \eqref{ProofPro2.4eq1} then follows from   \eqref{replacetheta}
by a simple application  of the Cauchy-Schwarz inequality. Note that
we may confine ourselves to considering $\lambda \in (1/\sqrt{N},1)$, since for $
\lambda \in (0, 1/\sqrt{N})$ the left side of \eqref{ProofPro2.4eq1} is upper
bounded by $2/N$.  To derive \eqref{ProofPro2.4eq1} for $
\lambda>1/\sqrt{N}$  we use  Lemma 2.3 from \cite{HK2012} and obtain

\begin{align}
&\lambda^2 \|\tilde{v}^{(1)}_{1, \lambda}- \hat{v}^{(1)}_{1, \lambda}\| \le  \lambda^2\frac{2 \sqrt{2} \|\tilde{c}^{(1)}(\lambda,
\cdot, \cdot )-\hat{c}^{(1)}(\lambda, \cdot, \cdot )\|}
{\tilde{\tau}^{(1)}_{1, \lambda}-\tilde{\tau}^{(1)}_{2, \lambda}},
\label{upperboundev}
\end{align}
where $\tilde{\tau}^{(1)}_{1, \lambda}$ and $\hat{\tau}^{(1)}_{2, \lambda}$ are the eigenvalues of covariance kernels $\tilde{c}^{(1)}(\lambda, \cdot, \cdot )$ defined in
  \eqref{optimalkernels1}.
We now consider numerator and denominator separately.

Beginning with the denominator we first notice that by consistency of the estimated
eigenvalues $\tilde{\tau}^{(1)}_{1, \lambda}$ and $\tilde{\tau}^{(1)}_{2,
\lambda}$ we have $\tilde{\tau}^{(1)}_{1, \lambda}-\tilde{\tau}
^{(1)}_{2, \lambda}=\tau_1^{(1)}-\tau_2^{(1)}+o_P(1)$. In particular, since $
\tau_1^{(1)}>\tau_2^{(1)}$ by Assumption \ref{assumption2} it is
bounded away from $0$ with probability converging to $1$. To see the consistency
of the eigenvalues, we use the following upper bound
\begin{eqnarray*}
&&|\tilde{\tau}_{1, \lambda}^{(1)} -\tau_1^{(1)}|\le  \| \tilde{c}^{(1)}(\lambda,  \cdot,\cdot) -c^{(1)} (\cdot,\cdot) \|=\mathcal{O}_P(N^{-1/4 }\log(N)),
\end{eqnarray*}
where the second equality follows from Lemma B.1 in the supplementary
material of \cite{ARS2018} and holds uniformly in $\lambda \in (1/\sqrt{N},1)$. Applying the same
argument to the second eigenvalue yields also consistency of $\tilde{\tau}^{(1)}_{2,
\lambda,}$.

In the proof of Proposition \ref{Proposition2} (step 1), we have already shown that the numerator on the right side of
\eqref{upperboundev}  is of order $o_P(1/\sqrt{N})$ and hence \eqref{replacetheta} follows.

We now turn to an investigation of the process
$$ K_N(\lambda) :=  \sqrt{N}\Big[ \int_0^1\left(\tilde{v}^{(1)}
_{1, \lambda}(t)-\tilde{v}^{(2)}_{1, \lambda}(t)\right)^2 \lambda^2
dt - \int_0^1\left( v^{(1)}_1(t)-v^{(2)}_1(t)\right)^2 \lambda^2 dt \Big] ,   $$
for which a simple calculation shows
\begin{eqnarray} \label{eq1pro1}
    K_N(\lambda)&=& \sqrt{N}\Big[ \int_0^1\left( \tilde{v}^{(1)}_{1,
\lambda}(t)-\tilde{v}^{(2)}_{1, \lambda}(t) -v^{(1)}_1(t)+v^{(2)}_1(t)
\right)^2 \lambda^2 dt\Big]\nonumber \\
&+& 2 \sqrt{N}\Big[ \int_0^1\left( v^{(1)}_1(t)-v^{(2)}_1(t)
 \right) \left( \tilde{v}^{(1)}_{1, \lambda}(t)-\tilde{v}^{(2)}_{1, \lambda}(t) -v^{(1)}_1(t)+v^{(2)}_1(t)\right) \lambda^2 dt\Big]  \\
&=& 2\sqrt{N}\Big[ \int_0^1\left( v^{(1)}_1(t)-v^{(2)}_1(t)\right) \left( \tilde{v}^{(1)}_{1, \lambda}(t)-\tilde{v}^{(2)}_{1, \lambda}(t) -v^{(1)}_1(t)+v^{(2)}_1(t)\right) \lambda^2 dt\Big] +o_P(1).\nonumber
 \end{eqnarray}
For the second equality we have used Proposition 2.1 from \cite{auedettrice2019}. In
order to determine the limiting behavior of this expression, we make several
technically helpful transformations beginning with a linearization.\\
Similar calculations as in the proof of Proposition 2.3 in \cite{auedettrice2019} yield the  representation
 \begin{align}
K_N(\lambda)= 2\lambda \sqrt{N} \Big( \frac{1}{\lfloor N \theta_0 \rfloor}
\sum_{n=1}^{\lfloor \lfloor N \theta_0 \rfloor \lambda \rfloor} \bar{X}_n^{(1)} +
\frac{1}{\lfloor N (1-\theta_0) \rfloor} \sum_{n= \lfloor N \theta _0\rfloor +1}^{ \lfloor
N \theta_0 \rfloor +\lfloor \lfloor N (1-\theta_0) \rfloor \lambda \rfloor} \bar{X}
_n^{(2)}\Big) +o_P(1), \label{eq1step2}
\end{align}
where the random variables $\bar{X}_n^{(1)}$ and $\bar{X}_n^{(2)}$ are defined by
\begin{eqnarray}
 \bar{X}_n^{(i)}&:=& \int_0^1\tilde{X}_n^{(i)}(s,t) f^{(i)}(s,t) ds dt \quad i=1,2, \\
 \label{xbar}
\tilde{X}_n^{(i)}(s,t)&=& X_n^{(i)}(s)X_n^{(i)}(t)-\mathbb{E}[X_n^{(i)}(s)X_n^{(i)}(t)] \quad
i=1,2 \label{Xtilde}
\end{eqnarray}
and
$$f^{(i)}(s,t)=-v_1^{(i)}(s)\sum_{k \ge 2} \frac{v_k^{(i)}(t)}{\tau_1^{(i)}-\tau_k^{(i)}}
\int_0^1v_k^{(i)}(u)v_1^{(3-i)}(u)du.$$
Notice that $f^{(i)}\in L^2[0,1]$. Weak convergence of the process $\{K_N(\lambda)\}_{\lambda \in [0,1]}$ defined in \eqref{eq1step2} follows from weak
 convergence of the two dimensional process
\begin{eqnarray}
\Big \{ \frac{\sqrt{N} \lambda}{\lfloor N \theta_0 \rfloor} \sum_{n=1}^{\lfloor \lfloor N
\theta_0 \rfloor \lambda \rfloor} \bar{X}_n^{(1)} , \frac{\sqrt{N} \lambda}{\lfloor N (1-
\theta_0) \rfloor} \sum_{n= \lfloor N \theta _0\rfloor +1}^{ \lfloor N \theta_0 \rfloor +
\lfloor \lfloor N (1-\theta_0) \rfloor \lambda\rfloor} \bar{X}_n^{(2)} \Big\}_{\lambda
\in [0,1]} \label{componentprocess}
 \end{eqnarray}
and the
continuous mapping theorem. \\
 Similar arguments as given in \cite{auedettrice2019} show that the components
of \eqref{componentprocess}
\begin{eqnarray}
\Big( \lambda \frac{\sqrt{N}}{\lfloor N \theta_0 \rfloor} \sum_{n=1}^{\lfloor \lfloor N
\theta_0 \rfloor \lambda \rfloor} \bar{X}_n^{(1)} \Big)_{\lambda  \in [0,1]} \quad
\textnormal{and} \quad \Big( \lambda\frac{\sqrt{N}}{\lfloor N (1-\theta_0) \rfloor}
\sum_{n= \lfloor N \theta _0\rfloor +1}^{ \lfloor N \theta_0 \rfloor +\lfloor \lfloor N (1-
\theta_0) \rfloor \lambda \rfloor} \bar{X}_n^{(2)}\Big)_{\lambda  \in [0,1]}
\nonumber
\end{eqnarray}
converge weakly to stochastic processes of the form $\sigma_D^{(1)} \sqrt{\theta_0}\lambda
\mathbb{B}_1(\lambda) $ and $ \sigma_D^{(2)}\sqrt{1- \theta_0}\lambda \mathbb{B}_2
(\lambda) $ for some suitable constants $\sigma_D^{(1)}$ and $\sigma_D^{(2)}$ (see equation \ref{longrunvariance1} below), where $\mathbb{B}_1$ and $\mathbb{B}_2$ are independent, standard Brownian motions. In
particular both processes are asymptotically tight  and consequently the vector in
\eqref{componentprocess} is also asymptotically tight. To complete the proof of weak
convergence of \eqref{componentprocess},  it therefore remains to prove the
convergence of the finite dimensional distributions.

For this purpose we replace the random variables $\tilde{X}_n^{(i)}$ in \eqref{xbar} by their  $m$-dependent
analogues
\begin{equation}
\bar{X}_{n,m}^{(i)}:= \int_0^1\tilde{X}_{n,m}^{(i)}(s,t) f^{(i)}(s,t) ds dt \quad i=1,2
\label{xmbar}
\end{equation}
where $$\tilde{X}_{n,m}^{(i)} (s,t):=X_{n,m}^{(i)} (s) X_{n,m}^{(i)}(t)-\mathbb{E}[X_{n,m}^{(i)}(s)X_{n,m}^{(i)}(t)] \qquad (i=1,2)$$
and the variables $X_{n,m}^{(i)}$ are defined in \eqref{m-dependentRV}.
The norm of
$$\frac{1}{\sqrt{N}}\sum_{n=1}^{\lfloor \lambda\lfloor N \theta_0 \rfloor \rfloor}
(\bar{X}_n^{(1)}-\bar{X}_{n,m}^{(1)})=\frac{1}{\sqrt{N}}\sum_{n=1}^{\lfloor \lambda
\lfloor N \theta_0 \rfloor \rfloor} \int_0^1\int_0^1(\tilde{X}_{n}^{(1)}(s,t)-\tilde{X}
_{n,m}^{(1)}(s,t))f^{(1)}(s,t)dsdt$$
is   bounded by
 \begin{align*}
&\max_{k=1,...,\lfloor N \theta_0\rfloor } \Big\| \frac{1}{\sqrt{N}} \sum_{n=1}^k
\big(X_n^{(1)} \otimes X_n^{(1)} - \mathbb{E}X_n^{(1)} \otimes X_n^{(1)}-X_{n,m}^{(1)}
\otimes X_{n,m}^{(1)} + \mathbb{E}X_{n,m}^{(1)} \otimes X_{n,m}^{(1)}\big)\Big\|
 \|f^{(1)}\|\\
 =&\max_{k=1,...,\lfloor N \theta_0\rfloor } \Big\| \frac{1}{\sqrt{N}} \sum_{n=1}^k
 \big(X_n^{(1)} \otimes X_n^{(1)}  -X_{n,m}^{(1)} \otimes X_{n,m}^{(1)}\big) \Big\|\|f^{(1)}\|
\end{align*}
and consequently converges to $0$ in probability according to
\eqref{berkeslemma21} if $m\to \infty$. The case $i=2$ can be treated
analogously. Therefore it is sufficient to prove the  convergence of the
finite dimensional distributions of the vector %Also works for Y;sum_k^N=sum_1^N-sum_1^k
$$ \bigg( \frac{\sqrt{N} \lambda}{\lfloor N \theta_0 \rfloor} \sum_{n=1}^{\lfloor
\lfloor N \theta_0 \rfloor \lambda \rfloor} \bar{X}_{n,m}^{(1)} , \frac{\sqrt{N} \lambda}
{\lfloor N (1-\theta_0) \rfloor} \sum_{n= \lfloor N \theta _0\rfloor +1}^{ \lfloor N
\theta_0 \rfloor +\lfloor \lfloor N (1-\theta_0) \rfloor \lambda \rfloor} \bar{X}_{n,m}
^{(2)} \bigg)_{\lambda \in [0,1]}, $$
which can be shown in the same way as in the proof of Step 2 in Proposition \ref{proposition2a}. Finally, we define
\begin{align}
(\sigma^{(i)}_D)^2= \mathbb{V}ar(\bar{X_0}^{(i)})+2\sum_{l \ge 1} \mathbb{C}ov
(\bar{X}_{0}^{(i)},\bar{X}_{l}^{(i)})>0 \qquad (i=1,2) \label{longrunvariance1}
\end{align}
(these quantities are positive
  by assumption) and obtain
\begin{align*}
&\Big( \frac{\sqrt{N} \lambda}{\lfloor N \theta_0 \rfloor} \sum_{n=1}^{\lfloor \lfloor N \theta_0 \rfloor \lambda \rfloor} \bar{X}_n^{(1)} , \frac{\sqrt{N} \lambda}{\lfloor N (1-\theta_0) \rfloor} \sum_{n= \lfloor N \theta _0\rfloor +1}^{ \lfloor N \theta_0 \rfloor +\lfloor \lfloor N (1-\theta_0) \rfloor \lambda \rfloor} \bar{X}_n^{(2)} \Big)_{ \lambda \in [0,1]} \\
\to_D& \left(   \sigma_D^{(1)}\lambda\mathbb{B}_1(\lambda)/\sqrt{\theta_0}, \sigma_D^{(2)} \lambda\mathbb{B}_2(\lambda)/\sqrt{1-\theta_0} \right)_{ \lambda \in [0,1]}.
\end{align*}
The continuous mapping theorem gives
\begin{align*}
& \left( \sqrt{N}\left[ \int_0^1\left(\hat{v}^{(1)}_{1,  \lambda}(t)-\hat{v}^{(2)}_{1,  \lambda}(t)\right)^2 \lambda^2 dt - \int_0^1\left( v^{(1)}_1(t)-v^{(2)}_1(t)\right)^2 \lambda^2 dt \right]\right)_{\lambda \in [0,1]} \\
 \to_D & \left( 2[\lambda\sigma_D^{(1)}/\sqrt{\theta_0} \mathbb{B}_1(\lambda)+\lambda\sigma_D^{(2)}/\sqrt{1-\theta_0}\mathbb{B}_2(\lambda)]\right)_{\lambda \in [0,1]}.
\end{align*}

Now the same steps as in the proof of Theorem \ref{theorem1} yield the desired result. \hfill $\Box$

\section{Technical details and supplementary results}
\label{appB}
 \def\theequation{B.\arabic{equation}}
\setcounter{equation}{0}

\subsection{Weak convergence of the covariance kernel}
In this section we provide an adaption of   Theorem 1.1  in \cite{Berkes2013} to the estimation of covariance kernels. Let $(X_n)_{n \in \mathbb{Z}}$ be a sequence of random functions satisfying Assumption \ref{assumption1} and consider   the sequential process
\begin{equation}
 S_N(x,s,t):= \frac{1}{\sqrt{N}}\sum_{n=1}^{\lfloor N x\rfloor} \left\{X_n(s) X_n(t)-\mathbb{E}X_n(s) X_n(t) \right\} \label{S_N}.
 \end{equation}
Thus we are interested in a sum of random elements $X_n \otimes X_n \in L^2([0,1]\times [0,1])$.
These products can be approximated by
products of the $m$-dependent random functions $X_{n,m} \otimes  X_{n,m}$, where  the
random variables  $X_{n,m}$ are  defined in Assumption \ref{assumption1} (note that $X_n$ and $X_{n,m}$ have the same distribution).
Using  Assumption \ref{assumption1}  and the notation $\delta'=\delta/2$ and $\kappa'=\kappa/2$
we obtain for a suitable constant $K>0$
\begin{eqnarray*}
&&\sum_{m \ge 1} \left( \mathbb{E} \|X_n \otimes X_n-\mathbb{E}X_n \otimes X_n -X_{n,m} \otimes X_{n,m}+\mathbb{E}X_{n,m} \otimes X_{n,m}\|^{2+\delta'}\right)^{1/\kappa'}\\
&=& \sum_{m \ge 1} \left( \mathbb{E} \|X_n \otimes X_n -X_{n,m} \otimes X_{n,m}\|^{2+\delta'}\right)^{1/\kappa'}
 \\
& \le & K\sum_{m \ge 1}\left( \mathbb{E}\|X_n\|^{4+2\delta'} \mathbb{E} \|X_n-X_{n,m}\|^{4+2\delta'} \right)^{1/(2\kappa')}\\
&+&
K\sum_{m\ge 1}\left( \mathbb{E}\|X_{n,m}\|^{4+2\delta'} \mathbb{E} \|X_n-X_{n,m}\|^{4+2\delta'} \right)^{1/(2\kappa')}\\
&=&2K\left( \mathbb{E}\|X_0\|^{4+2\delta'}\right)^{1/(2\kappa')} \sum_{m \ge 1}\left( \mathbb{E} \|X_n-X_{n,m}\|^{4+2\delta'} \right)^{1/(2\kappa')}<\infty.                                                                                                                                                                                                                                                                                                                                                                                                              \end{eqnarray*}
This consideration demonstrates that we have the same approximation properties as required in \cite{Berkes2013}  for the random functions $X_n \otimes X_n$.

The corresponding ``long run covariance kernel'' for the sequence $(X_n \otimes X_n)_{n \in \mathbb{Z}}$ can be defined as
\begin{align*}
 C(q,r,s,t):=\mathbb{E}[X_0(q)X_0(r)X_0(s)X_0(t)]+&\sum_{n \ge 1} \mathbb{E}[X_0(q)X_0(r)X_n(s)X_n(t)]\\
 +&\sum_{n \ge 1} \mathbb{E}[X_n(q)X_n(r)X_0(s)X_0(t)].
\end{align*}
By analogous arguments as presented in Lemma 2.2  of  \cite{Berkes2013}, it can be observed that $C$ is square integrable and thus defines a Hilbert-Schmidt operator (see e.g. \cite{Bu1998} p.168). It thus follows that there exists a spectral decomposition of the integral operator with kernel $C$. Let us call its eigenfunctions $\Psi_1, \Psi_2,...$ and its corresponding eigenvalues $\Lambda_1, \Lambda_2,...$. With this eigensystem we can define the Gaussian process
\begin{align*}
\Gamma(x,s,t):=\sum_{1 \le l} \Lambda_{l} \mathbb{B}_{l}(x)\Psi_{l}(s,t),
 \end{align*}
 where $\mathbb{B}_{l}$ are independent Brownian motions for all $l\ge 1$ and $x \in [0,1]$. We now state an analogue of Berkes' Theorem 1.1. The proof runs along the same lines as in \cite{Berkes2013} and is therefore omitted.

\begin{theorem} \label{thmapp1}
Suppose Assumptions \ref{assumption1} hold. Then there exists for every $N$ a Gaussian process $\Gamma_N=_D\Gamma$ such that
$$\sup_x \int_0^1\int_0^1\left( S_N(x, t, s) - \Gamma_N(x, s,t)\right)^2 ds dt =o_P(1).$$
 \end{theorem}

\subsection{Eigenvalue-expansion}

In this section we investigate a stochastic linearization of the estimated eigenvalues of the empirical covariance operator. For this purpose let $(X_n)_{n \in \mathbb{Z}}$ be a stationary, functional time series, with vanishing mean function, that complies to the Assumptions \ref{assumption1}  and  \ref{assumption2}. We call the corresponding covariance kernel $c$, its eigenvalues $\tau_1 \ge \tau_2 \ge...$ and its eigenfunctions $v_1, v_2,...$. For the data sample $X_1,...,X_N$ we define the sequential estimator of the covariance kernel
\begin{equation}
\hat c(\lambda, s,t):= \frac{1}{\lfloor N \lambda \rfloor} \sum_{n=1}^{\lfloor N \lambda \rfloor} X_n(s) X_n(t). \label{ApBkernel}
\end{equation}

Its eigenvalues are denoted by $\hat \tau_{1, \lambda} \ge \hat \tau_{2, \lambda} \ge...$ and its eigenfunctions by $\hat v_{1,\lambda}, \hat v_{2,\lambda},...$.

\begin{proposition} \label{propositionev}
Under the above assumptions we have for any $j \in \{1,...,p\}$
\begin{align}
\sup_{\lambda \in [0,1] } \Big|\lambda (\th-\t)-\frac{1}{\sqrt{N }} \int_0^1 \hat{Z}_N (s,t,\lambda)  v_j (s)  v_j (t) ds dt \Big|=o_P(1/\sqrt{N}), \label{eq1}
\end{align}
where
$$ \hat{Z}_N (s,t,\lambda):= \frac{1}{\sqrt{N}} \sum_{n=1}^{\lfloor N \lambda \rfloor} (X_n (s)X_n (t) - c (s,t)).$$
Furthermore, with $\kappa$ defined in Assumption \ref{assumption1}, 4. we have
\begin{equation} \label{H4}
 \sup_{\lambda \in [0,1] } |\sqrt{\lambda} (\th-\t)|=\mathcal{O}_P(\log(N)^{1/\kappa}/\sqrt{N}).
 \end{equation}
\end{proposition}

\begin{proof} 
We begin by noticing that
\begin{equation}
 \sup_{\lambda \in [0,1] }\Big| \frac{\lfloor N \lambda \rfloor}{N}\int_0^1 \int_0^1 \cx \vjs \vjt ds dt-  \lambda \t \Big| = \Big| \frac{\lfloor N \lambda\rfloor}{N} \t- \lambda \t\Big|= \mathcal{O}(1/N). \label{ApB2eq1}
\end{equation}
Recalling the definition of $\cxh$as defined in \eqref{ApBkernel}, 
we use \eqref{ApB2eq1} to rewrite \eqref{eq1}  as
\begin{align*}
& \sup_{\lambda \in [0,1] } \lambda \Big | \th -\int_0^1 \int_0^1 \cxh \vjs \vjt ds dt  \Big|+\mathcal{O}_P(1/N)\\
=& \sup_{\lambda \in [0,1] } \lambda \Big| \int_0^1 \int_0^1 \cxh [\vhs \vht-\vjs \vjt] ds dt  \Big|+\mathcal{O}_P(1/N)\\
\le & \sup_{\lambda \in [0,1] } \lambda \Big| \int_0^1 \int_0^1 \cxh [\vhs \vht-\vhs \vjt] ds dt  \Big|\\
+&\sup_{\lambda \in [0,1] }  \lambda \Big| \int_0^1 \int_0^1 \cxh [\vhs \vjt-\vjs \vjt] ds dt  \Big|+\mathcal{O}_P(1/N)\\
=&A+B+ \mathcal{O}_P(1/N),
\end{align*}
where the last equality defines the terms $A$ and $B$ in an obvious way.
We now investigate the terms $A$ and $B$ separately. For the term $A$ we observe that $\hat v_j$ is the eigenfunction of the integral operator associated with $\hat c $, which gives
\begin{align*}
A=& \sup_{\lambda \in [0,1] } \lambda\left| \th \int_0^1 \vht [\vht -\vjt]dt\right|\\
=& \sup_{\lambda \in [0,1] } \th  \left\| \hat{v}_{j, \lambda} -v_j  \right\|^2 \lambda/2 = o_P(1/\sqrt{N}).
\end{align*}
Here the second equality follows by the parallelogram law and  in the last step we used the estimate
\begin{equation}\label{det19}
\sup_{\lambda \in [0,1]} \left\| \hat{v}_{j, \lambda} -v_j  \right\|^2  \lambda = {o}_P (1/\sqrt{N})
\end{equation}
 (see Proposition 2.1 in \cite{auedettrice2019}) and
$$ \sup_{\lambda \in [0,1] } \th \le \sup_{\lambda \in [0,1] } (|\th-\t|+\t)=\mathcal{O}_P(1), $$
according to  Lemma B.1 in the supplementary material for \citep{ARS2018}. Turning to   term $B$ we see that
\begin{align*}
B \le &  \sup_{\lambda \in [0,1] }\lambda  \Big| \int_0^1 \int_0^1 \left[\cxh - \cx \right]\vjt [\vhs-\vjs]ds dt\Big| \\
+& \sup_{\lambda \in [0,1] }\lambda \Big| \int_0^1 \int_0^1  \cx \vjt [\vhs-\vjs]ds dt\Big| =: R_1+R_2,
\end{align*}
where the last equality defines the random variables $R_1$ and $R_2$ in an obvious manner.
For the term $R_1$ we obtain by the   Cauchy-Schwarz inequality
\begin{align*}
R_1 \le & \sup_{\lambda \in [0,1] } \sqrt{\lambda}\left\| \hat{c} (\cdot, \cdot, \lambda)-c \right\| \|v_j \| (\sqrt{\lambda}\|\hat{v}_{j, \lambda} -v_j \| ).
\end{align*}
Again by Lemma B.1 from the supplement of  \citep{ARS2018} we observe that
$$ \sup_{\lambda \in [0,1] } \sqrt{\lambda}\left\| \hat{c} (\cdot, \cdot, \lambda)-c \right\| = \mathcal{O}(\log^{1/\kappa}/\sqrt{N}).$$
\eqref{det19} shows that
 $R_1 = o_P(1/\sqrt{N})$.  We use  similar arguments and obtain
\begin{align*}
R_2=& \sup_{\lambda \in [0,1] } \lambda \t \left|\int_0^1 \vjt [\vht - \vjt] dt\right| =  \sup_{\lambda \in [0,1] } \left( \lambda\|\hat{v}_{j, \lambda} -v_j \|^2 \right) \t/2=o_P(1/\sqrt{N}).
\end{align*}
Combining these considerations proves the first assertion \eqref{eq1}.
For  a proof of  \eqref{H4} we note that
$$\sup_{\lambda \in [0,1] } |\sqrt{\lambda} (\th-\t)| \le\sup_{\lambda \in [0,1] } \Big\|\frac{\sqrt{\lambda}}{{\lfloor N \lambda \rfloor} }\sum_{n=1}^{\lfloor N \lambda \rfloor} X _n\otimes X _n-c \Big\|+\mathcal{O}_P(N^{-1})=\mathcal{O}_P(\log(N)^{1/\kappa}N^{-1/2}).$$
The first inequality follows from  bounding the eigenvalue distance by the operator distance and this again by the $L^2$-distance of the kernels. The second one follows by a Lemma B.1 in the supplementary material for \citep{ARS2018}.

\end{proof}

\noindent {\bf Acknowledgements}
This work has been supported in part by the
Collaborative Research Center ``Statistical modeling of nonlinear
dynamic processes'' (SFB 823, Project A1,  C1) of the German Research Foundation (DFG). 
%\nocite{*}
% \addcontentsline{toc}{chapter}{\bibname}
\bibliographystyle{apalike}

\begin{thebibliography}{}

\bibitem[Aston and Kirch, 2012]{aston2012}
Aston, J.~A. and Kirch, C. (2012).
\newblock Detecting and estimating changes in dependent functional data.
\newblock {\em Journal of Multivariate Analysis}, 109:204--220.

\bibitem[Aue et~al., 2019]{auedettrice2019}
Aue, A., Dette, H., and Rice, G. (2019).
\newblock Two-sample tests for relevant differences in the eigenfunctions of
  covariance operators.
\newblock {\em arXiv:1909.06098v1}.

\bibitem[Aue et~al., 2015]{adh}
Aue, A., Dubart~Nourinho, D., and H\"ormann, S. (2015).
\newblock On the prediction of stationary functional time series.
\newblock {\em Journal of the American Statistical Association}, 110:378--392.

\bibitem[Aue et~al., 2009]{AUE2009}
Aue, A., Gabrys, R., Horv{\'{a}}th, L., and Kokoszka, P. (2009).
\newblock Estimation of a change-point in the mean function of functional data.
\newblock {\em Journal of Multivariate Analysis}, 100(10):2254--2269.

\bibitem[Aue et~al., 2018]{ARS2018}
Aue, A., Rice, G., and S{\"{o}}nmez, O. (2018).
\newblock Detecting and dating structural breaks in functional data without
  dimension reduction.
\newblock {\em Journal of the Royal Statistical Society, Ser.\ B},
  80(3):509--529.

\bibitem[Berger and Delampady, 1987]{bergdela1987}
Berger, J.~O. and Delampady, M. (1987).
\newblock Testing precise hypotheses.
\newblock {\em Statistical Science}, 2:317--335.

\bibitem[Berk, 1973]{Be1973}
Berk, K.~N. (1973).
\newblock A central limit theorem for $m$-dependent random variables with
  unbounded $m$.
\newblock {\em Annals of Probability}, 1(2):352--354.

\bibitem[Berkes et~al., 2013]{Berkes2013}
Berkes, I., Horv{\'{a}}th, L., and Rice, G. (2013).
\newblock Weak invariance principles for sums of dependent random functions.
\newblock {\em Stochastic Processes and their Applications}, 123(2):385--403.

\bibitem[Bosq, 2000]{bosq2000}
Bosq, D. (2000).
\newblock {\em Linear Processes in Function Spaces}.
\newblock Springer-Verlag, New York.

\bibitem[B{\"{u}}cher et~al., 2019]{BDH2018}
B{\"{u}}cher, A., Dette, H., and Heinrichs, F. (2019).
\newblock Detecting deviations from second-order stationarity in locally
  stationary functional time series.
\newblock {\em Annals of the Institute of Statistical Mathematics, to appear,
  arXiv:1808.04092}.

\bibitem[Bump, 1996]{Bu1998}
Bump, D. (1996).
\newblock {\em Automorphic Forms and Representations.}
\newblock Cambridge University Press, Cambridge.

\bibitem[Dauxois et~al., 1982]{dauxois1982asymptotic}
Dauxois, J., Pousse, A., and Romain, Y. (1982).
\newblock Asymptotic theory for the principal component analysis of a vector
  random function: some applications to statistical inference.
\newblock {\em Journal of multivariate analysis}, 12(1):136--154.

\bibitem[Dette et~al., 2019]{dettkokaue2019}
Dette, H., Kokot, K., and Aue, A. (2019).
\newblock Functional data analysis in the {B}anach space of continuous
  functions.
\newblock {\em Annals of Statistics, to appear, arXiv:1710.07781}.

\bibitem[Dette et~al., 2018]{DKV2018}
Dette, H., Kokot, K., and Volgushev, S. (2018).
\newblock Testing relevant hypotheses in functional time series via
  self-normalization.
\newblock {\em arXiv:1809.06092}.

\bibitem[Dette and Wied, 2016]{detteWied2015}
Dette, H. and Wied, D. (2016).
\newblock Detecting relevant changes in time series models.
\newblock {\em Journal of the Royal Statistical Society: Series B},
  78:371--394.

\bibitem[Ferraty and Vieu, 2010]{FerratyVieu2010}
Ferraty, F. and Vieu, P. (2010).
\newblock {\em Nonparametric {F}unctional {D}ata {A}nalysis}.
\newblock Springer-Verlag, New York.

\bibitem[Fremdt et~al., 2014]{fremdtetal2014}
Fremdt, S., Horv{\'{a}}th, L., Kokoszka, P., and Steinebach, J. (2014).
\newblock Functional data analysis with an increasing number of projections.
\newblock {\em Journal of Multivariate Analysis}, 124:313--332.

\bibitem[Gao et~al., 2019]{GAO2019232}
Gao, Y., Shang, H.~L., and Yang, Y. (2019).
\newblock High-dimensional functional time series forecasting: An application
  to age-specific mortality rates.
\newblock {\em Journal of Multivariate Analysis}, 170:232 -- 243.
\newblock Special Issue on Functional Data Analysis and Related Topics.

\bibitem[Hall and Hosseini-Nasab, 2006]{Hall2006}
Hall, P. and Hosseini-Nasab, M. (2006).
\newblock On properties of functional principal components.
\newblock {\em Journal of the Royal Statistical Society, Ser.\ B}, 68:109--126.

\bibitem[Hodges and Lehmann, 1954]{Hodges1954}
Hodges, J. and Lehmann, E. (1954).
\newblock Testing the approximate validity of statistical hypotheses.
\newblock {\em Journal of the Royal Statistical Society, Ser.\ B},
  16(2):261--268.

\bibitem[Horv{\'{a}}th and Kokoszka, 2012]{HK2012}
Horv{\'{a}}th, L. and Kokoszka, P. (2012).
\newblock {\em Inference for Functional Data with Applications.}
\newblock Springer Series in Statistics, New York.

\bibitem[Hsing and Eubank, 2015]{hsingeubank2015}
Hsing, T. and Eubank, R. (2015).
\newblock {\em Theoretical Foundations of Functional Data Analysis, with an
  Introduction to linear Operators}.
\newblock Wiley, New York.

\bibitem[K{\"{o}}nig, 1986]{Koenig1986}
K{\"{o}}nig, H. (1986).
\newblock {\em Eigenvalue Distribution of Compact Operators}.
\newblock Birkh{\"{a}}user, Basel.

\bibitem[Ramsay and Silverman, 2005]{RS2005}
Ramsay, J.~O. and Silverman, B.~W. (2005).
\newblock {\em Functional Data Analysis}.
\newblock Springer, New York.

\bibitem[Rice and Shum, 2019]{RS2019}
Rice, G. and Shum, M. (2019).
\newblock Inference for the lagged cross-covariance operator between functional
  time series.
\newblock {\em Journal of Time Series Analysis}, 40(5):665--692.

\bibitem[Shang, 2014]{Shang2014ASO}
Shang, H.~L. (2014).
\newblock A survey of functional principal component analysis.
\newblock {\em AStA Advances in Statistical Analysis}, 98:121--142.

\bibitem[Shang, 2017]{SHANG2017184}
Shang, H.~L. (2017).
\newblock Functional time series forecasting with dynamic updating: An
  application to intraday particulate matter concentration.
\newblock {\em Econometrics and Statistics}, 1:184 -- 200.

\bibitem[Shao and Zhang, 2010]{shao2010}
Shao, X. and Zhang, X. (2010).
\newblock Testing for change points in time series.
\newblock {\em Journal of the American Statistical Association},
  105(491):1228--1240.

\bibitem[Stoehr et~al., 2019]{stoehr2019}
Stoehr, C., Aston, J., and Kirch, C. (2019).
\newblock Detecting changes in the covariance structure of functional time
  series with application to f{MRI} data.
\newblock {\em arXiv:1903.00288}.

\bibitem[van~der Vaart and Wellner, 1996]{VW1996}
van~der Vaart, A. and Wellner, J. (1996).
\newblock {\em Weak convergence and Empirical Processes. With Applications to
  Statistics}.
\newblock Springer Series in Statistics, New York.

\bibitem[Weidmann, 1980]{JW1980}
Weidmann, J. (1980).
\newblock {\em Linear Operators in Hilbert spaces}.
\newblock Springer, New York.

\bibitem[Wellek, 2010]{wellek2010}
Wellek, S. (2010).
\newblock {\em Testing Statistical Hypotheses of Equivalence and
  Noninferiority}.
\newblock CRC Press, Boca Raton, FL, second edition.

\bibitem[Zhang and Shao, 2015]{zhang2015}
Zhang, X. and Shao, X. (2015).
\newblock Two sample inference for the second-order property of temporally
  dependent functional data.
\newblock {\em Bernoulli}, 21(2):909--929.

\end{thebibliography}
\setlength{\bibsep}{1pt}
\begin{small}

\end{small}

 \end{document}